\documentclass{amsart}
\usepackage[utf8]{inputenc}
\usepackage{comment}
\usepackage{amscd}
\usepackage{eucal}
\usepackage{todonotes}
\usepackage{amssymb}
\usepackage{pstricks}
\usepackage{graphicx}
\usepackage{lineno}
\usepackage{subcaption}
\usepackage{verbatim}

\newtheorem{definition}{Definition}
\newtheorem{theorem}{Theorem}

\newtheorem{lemma}{Lemma}
\newtheorem{corollary}{Corollary}
\newtheorem{proposition}{Proposition}
\newtheorem{remark}{Remark}

\newcommand{\R}{\mathbb{R}}

\numberwithin{equation}{section}

\author[Huerta]{Ignacio Huerta}
\author[Monzón]{Pablo Monzón}
\author[Robledo]{Gonzalo Robledo}
\email{ignacio.huertan@usm.cl, monzon@fing.edu.uy, grobledo@uchile.cl}
\address{Departamento de Matem\'atica, Universidad T\'ecnica Federico Santa Mar\'ia, Casilla 110-V, Valpara\'iso, Chile.}
\address{Facultad de Ingenier\'ia -- Universidad de la Rep\'ublica, Julio Herrera y Ressing 565, Montevideo, Uruguay.}
\address{Departamento de Matem\'aticas -- Universidad de Chile, Casilla 653, Las Palmeras 3425, \~Nu\~noa -- Santiago, Chile.}
\title[Controllability and feedback stabilization]{Controllability and feedback stabilizability in a nonuniform framework}
\thanks{This research has been supported by the grants FONDECYT Regular 1210733 and FONDECYT Postdoctorado 3210132}

\begin{document}

\begin{abstract}
We propose a new controllability property for linear non\-au\-to\-no\-mous control systems in finite dimension: the nonuniform complete con\-tro\-lla\-bility, which is halfway between the classical Kalman's properties of complete con\-tro\-lla\-bility and uniform complete controllability. This new concept has a strong linkage; as we prove; with the property of nonuniform bounded growth for the corresponding plant. In addition, we also prove that if a control system is nonuniformly completely controllable and its plant (uncontrolled part) has the property of nonuniform bounded growth, then there exist a linear feedback control leading to a nonuniformly exponentially stable closed--loop system.
\end{abstract}

\maketitle
 

\section{Introduction}
The properties of controllability and feedback stabilizability are fundamental topics of control theory and rely on intrisic properties of the involved dynamical systems. This work provides new results describing the connection between these two concepts for a particular case of the nonautonomous linear control systems:
\begin{equation}
\label{control1}
    \dot{x}(t)=A(t)x(t)+B(t)u(t),
\end{equation}
where $t\mapsto A(t)\in M_{n\times n}(\mathbb{R})$ and $t\mapsto B(t)\in M_{n\times p}(\mathbb{R})$ are matrix valued functions for any $t\geq 0$, having orders  $n\times n$ and $n\times p$ respectively. In addition, 
$A(\cdot)$ and $B(\cdot)$ are measurable and bounded on finite intervals, $t\mapsto x(t)\in \mathbb{R}^{n}$ is known as the \textit{state vector} and piecewise continuous map $t\mapsto u(t)\in \mathbb{R}^{p}$ is the \textit{open--loop control} or \textit{input}. Finally, when considering the null input  $u(t) \equiv 0$, this enable us to call the linear system 
\begin{equation}
\label{lin}
\dot{x}=A(t)x
\end{equation}
as the \textit{plant} or
the \textit{uncontrolled part} of the control system (\ref{control1}). The transition matrix of (\ref{lin}) is denoted by $\Phi_{A}(t,s)$. Now, any solution of (\ref{lin}) passing through
$x_{0}$ at $t=t_{0}$, is denoted by $t\mapsto x(t,t_{0},x_{0})=\Phi_{A}(t,t_{0})x_{0}$. In addition, given an input $u$, any solution of (\ref{control1}) passing through the initial condition $x_{0}$ at time $t=t_{0}\geq 0$ will be denoted by $t\mapsto x(t,t_{0},x_{0},u)$. The solution of \eqref{control1} is given by the expression:
\begin{equation}\label{eq:solution}
x(t,t_{0},x_0,u)=\Phi(t,t_{0})x_0+\int_{t_0}^t\Phi(t,\tau)B(\tau)u(\tau)d\tau.   
\end{equation}

A linear control system is controllable if, roughly speaking, for any initial condition $x_{0}$ there exists at least one input $u$ that drives it to the origin in finite time; a formal definition will be stated in the next section.

If $A(\cdot)$ and $B(\cdot)$ are constant matrices, namely, when (\ref{control1}) is an autonomous linear control system, the controllability has been studied in depth. In particular, it has been established that the controllabilty is equivalent to the pro\-per\-ty of \textit{feedback stabilizability}: there exists a linear feedback input $u=-Fx$ with $F\in M_{p,n}(\mathbb{R})$ such that the closed loop system
\begin{equation}\label{eq:closedloop}
 \dot{x}=[A-BF]x   
\end{equation}
is uniformly asymptotically stable, and we refer to \cite{Anderson-1990,Chen,CLSD,Kalman-60,ZDG} for details.

 Contrarily to the autonomous framework, there is no direct equivalence between con\-tro\-lla\-bi\-li\-ty and feedback stabilizability for nonautonomous linear control systems. This is due to the fact that the above properties are no  univocally extended for the nonautonomous framework: there are numerous notions of controllability and asymptotic stability, which leads to different possible results relating them, as will be described in the next sections. In consequence, the problem of feedback stabilizability for linear nonautonomous control systems must be addressed by taking into account this multiplicity of properties.

\subsection{Novelty of this work}
The present work introduces and characterizes a controllability property for nonautonomous control systems: the \textit{nonuniform complete controllability} (NUCC), which is more general than the classical Kalman's  \textit{uniform complete controllability} (UCC) but more restrictive than mere \textit{complete controllability} (CC), both described in Section 2. We present examples of systems that shows the differences between these three classes.\\


The first novel result is the Lemma \ref{EKC}, which relates the classical UCC with the property of uniform bounded growth for nonautonomous linear systems (\ref{lin}), which is well known in the qualitative theory of linear nonautonomous systems but not so much in control theory.

Our second novel result is the Theorem \ref{p2l3}, which states
that the plant of a nonuniformly completely controllable system (\ref{lin}) must satisfy 
a condition that will be called as \textit{nonuniform Kalman's condition}, which is weaker than the classical uniform version. The proof is fa\-shio\-ned along the Kalman's approach \cite{Kalman} but the nonuniformities induce technical difficulties and forces a deeper treatment.\\

The third novel result concerns with feedback stabilization: 
Theorem \ref{T2} proves that if a linear control system (\ref{control1}) is NUCC and its plant has the property of nonuniform bounded growth property, it can be stabilized by a linear feedback input such that any solution of the closed loop system starting at $t=t_{0}$ is nonuniformly exponentially stable.

\subsection{Notations and basic settings}
Given $M\in M_{n}(\mathbb{R})$, $M^{T}$ is the transpose, $\operatorname{tr}(M)$ is the trace. The inner product of two vectors $x,y\in \mathbb{R}^{n}$ is denoted by
$\langle y, x\rangle=x^{T}y$ and the euclidean norm of a vector will be denoted by $|x|=\sqrt{\langle x,x\rangle}$. Given $A\in M_{n}(\mathbb{R})$, the matrix norm induced by $|\cdot|$ is
\begin{displaymath}
\|A\|=\sup\limits_{\eta\neq 0}\frac{|A\eta|}{|\eta|}=\sqrt{\lambda_{\max}(A^{T}A)},
\end{displaymath}
where $\lambda_{\max}(A^{T}A)$ is the maximum eigenvalue of $A^{T}A$, while its minimum eigenvalue
will be denoted by $\lambda_{\min}(A^{T}A)$.

A symmetric matrix $M=M^{T}\in M_{n}(\mathbb{R})$ is semi--positive definite if
$x^{T}Mx =\langle M x, x\rangle  \geq 0$ for any real vector $x\neq 0$, and this property will be denoted as $M \geq 0$.
In case that the above inequalities are strict, we say that $M=M^{T}$ is positive definite.

Given two matrices $M, N\in M_{n}(\mathbb{R})$, we write $M\leq N$ if  $N-M\geq0$ or equivalently, if $\langle M x, x\rangle \leq\langle N x, x\rangle$ for any $x \in \R^{n}$.

If $M\in M_{n}(\mathbb{R})$ is positive definite and $\kappa$ is a positive scalar such that $M\leq  \kappa I$ is verified,
it will be useful to recall that
\begin{equation}
\label{TRVP}
\operatorname{tr}(M)\leq n\kappa \quad \textnormal{and} \quad 0\leq \lambda_{\min}(M)\leq \lambda_{\max}(M) \leq \kappa.
\end{equation}

Finally, we will denote by $\mathcal{B}$ the set of real functions mapping bounded sets into bounded sets. 

\subsection{Structure of the article}
This article is on the crossroads of linear control systems theory and the theory of nonautonomous dynamical systems. In addition, the proof of our main results use methods and ideas which are current tools in nonautonomous dynamics but are not well known in control theory and; to make things more complicated; the classical control theory is not well known for scholars working on dynamical systems. In order to mitigate these mismatches, we estimated necessary and useful to write an encyclopaedic Section 2 focused to provide a common basis for both research communities.

 
 In section 3 we move into the nonuniform framework, introducing the new notion of nonuniform complete controllability (NUCC) and describe its main consequences. Finally, the section 4 provides sufficient conditions such that nonuniform complete controllability implies feedback stabilizability, where the closed loop system \eqref{eq:closedloop} is nonuniformly exponentially stable. 
 
\section{Controllability and Feedback stabilizability: basic notions and nonuniform preliminaries}
In this section we briefly recall the classical definitions and results of controllability and feedback stabilizability stated respectively in the seminal Kalman's paper \cite{Kalman}
and the article of Ikeda \textit{et al.} \cite{Ikeda}. Moreover, we will see that the uniform complete controllability is related with the property of uniform bounded growth while
the specificities of the feedback stabilizability can be understood by considering some types of exponential stabilities: uniform and nonuniform ones. There exist several ways to define these stabilities but we choose to define them as a limit case of the properties of uniform/nonuniform exponential dichotomy.


\subsection{Controllability}
The properties of \textit{controllability} were introduced by R. Kal\-man in \cite{Kalman} for the linear control system \eqref{control1} and we will briefly recall them by following \cite{Anderson,Sontag}:
\begin{definition}
The state $x_{0}\in \mathbb{R}^{n}$ of the control system \eqref{control1} is controllable at time $t_{0}\geq 0$, if there exists an input $u\colon [t_{0},t_{f}]\to \mathbb{R}^{p}$ such that $x(t_{f},t_{0},x_{0},u)=0$. In addition, the control system \eqref{control1} is: 
\begin{itemize}
\item[a)] \textbf{Controllable at time $t_{0}\geq 0$} if any state $x_{0}$ is controllable at time $t_{0}\geq 0$,
\item[b)] \textbf{Completely controllable} (CC) if it is controllable at any time $t_{0}\geq 0$.
\end{itemize}
\end{definition}
 
There exists a well known necessary and sufficient condition ensuring both controllability at time $t_{0}$ and complete controllability, which is stated in terms of the controllability gramian matrix, usually defined by
\begin{equation*}
W(a,b)=\displaystyle\int_{a}^{b}\Phi_{A}(a,s)B(s)B^{T}(s)\Phi_{A}^{T}(a,s)\;ds.
\end{equation*}

The control system (\ref{control1}) is {\it controllable at time $t_{0}\geq 0$} if and only if there exists
$t_{f}>t_{0}\geq 0$ such that $W(t_{0},t_{f})>0$. In addition, is {\it completely controllable} if and only if for 
any $t_{0}\geq 0$, there exists $t_{f}>t_{0}$ such that $W(t_{0},t_{f})>0$. We refer the reader to \cite{Anderson67,Kreindler} for a detailed description. Essentially, invertibility of the gramian allows us to construct an explicit input function that drives the system towards the origin. By considering  a given couple of initial state $x_0$ and initial time $t_0$, together with the following input $u^{*}\colon [t_{0},t_{f}]\to \mathbb{R}$ described by:
\begin{equation}
\label{input}
u^{*}(t)=-B^T(t)\Phi^T(t_0,t)W^{-1}(t_0,t_f)x_0,
\end{equation}
we can see by (\ref{eq:solution}) that the respective solution of (\ref{control1}) with $u=u^{*}$ satisfies:
\begin{displaymath}
\begin{array}{rl}
x(t_f,t_{0},x_{0},u^{*})&=\displaystyle\Phi(t_f,t_0)x_0-\int_{t_0}^{t_f}\Phi(t_f,\tau)B(\tau)B^T(\tau)\Phi^T(t_0,\tau)W^{-1}(t_0,t_f)x_0d\tau,\\\\
&=\displaystyle\Phi(t_f,t_0)x_0-\Phi(t_f,t_0)W(t_0,t_f) W^{-1}(t_0,t_f)x_0=0.
\end{array}
\end{displaymath}

It is important to emphasize that, in some references, the controllability gramian matrix is also denoted by 
\begin{equation*}
K(a,b)=\displaystyle\int_{a}^{b}\Phi_{A}(b,s)B(s)B^{T}(s)\Phi_{A}^{T}(b,s)\;ds,
\end{equation*}
which verifies the following properties
\begin{equation}
\label{GP}
\left\{\begin{array}{rcl}
K(a,b)&=&\Phi_{A}(b,a)W(a,b)\Phi_{A}^{T}(b,a) \\
W(a,b)&=&\Phi_{A}(a,b)K(a,b)\Phi_{A}^{T}(a,b),
\end{array}\right.
\end{equation}
and the above mentioned controllability condition can also be stated in terms of invertibility of $K(t_{0},t_{f})$.\\

A stronger property of controllability, also due to R. Kalman \cite{Kalman,Kalman69}, is given by the \textit{uniform complete controllability} (UCC). In this case, there exists a time $\sigma>0$ such that ``\emph{one
can always transfer $x$ to 0 and 0 to $x$ in a finite length $\sigma$ of time; moreover, such a transfer
can never take place using an arbitrarily small amount of control energy}'' \cite[p.157]{Kalman}. Kalman stated the property in terms of the gramian. We recall a slightly different version restricted to the positive half line:
\begin{definition}    The linear control system \eqref{control1} is said to be uniformly completely controllable on $[0,+\infty)$ if there exists a fixed constant $\sigma>0$ and positive numbers $\alpha_0(\sigma)$, $\beta_{0}(\sigma)$, $\alpha_1(\sigma)$ and $\beta_1(\sigma)$ such that the following relations hold for all $t\geq 0$: 
\begin{equation}
\label{UC1}
0<\alpha_0(\sigma)I\leq W(t,t+\sigma)\leq \alpha_1(\sigma)I,
\end{equation}
and
\begin{equation}
\label{UC2}
0<\beta_0(\sigma)I\leq K(t,t+\sigma) \leq \beta_1(\sigma)I.
\end{equation}
\end{definition}

By using (\ref{GP}), we can see that the condition (\ref{UC2}) is equivalent to
\begin{displaymath}
0<\beta_0(\sigma)I\leq  \Phi_{A}(t+\sigma,t)W(t,t+\sigma)\Phi_{A}^{T}(t+\sigma,t)  \leq \beta_1(\sigma)I,
\end{displaymath}
which is widely employed in the literature. Nevertheless, by following \cite[p.114]{ZA}, we adopted (\ref{UC2}) by its practical convenience. It is important to emphasize that, in several references, the explicit dependence of $\alpha_{i}$ and $\beta_{i}$ ($i=0,1$)
with respect to $\sigma$ is not considered\footnote{In Appendix \ref{app:energy}, we present some considerations about the control energy required to transfer the state $x_0$ to the origin.}.\\

The property of uniform complete controllability have noticeable consequences, which have been pointed out by R. Kalman
in \cite{Kalman} and are summarized by the following result, whose proof is sketched in \cite[p.157]{Kalman} (see also \cite{Silverman}):
\begin{proposition}
\label{Prop1}
If the linear control system \eqref{control1} is UCC then: 
\begin{itemize}
\item[i)] The inequalities \eqref{UC1}--\eqref{UC2} are also verified for any $\sigma'>\sigma$. 
\item[ii)] There exists a function $\alpha\colon [0,+\infty)\to (0,+\infty)$, with $\alpha(\cdot)\in\mathcal{B}$, such that the plant \eqref{lin} has a transition matrix verifying the property
\begin{equation}
\label{BG}
\|\Phi_{A}(t,s)\| \leq \alpha(|t-s|) \quad \textnormal{for all $t,s \in [0,+\infty)$}.
\end{equation}
\end{itemize}
\end{proposition}

The above result deserves some comments:
\begin{remark}
The statement i) of Proposition \ref{Prop1} implies the existence of positive functions
$\alpha_{0},\alpha_{1},\beta_{0},\beta_{1}\colon [\sigma,+\infty) \to (0,+\infty)$ related to
\eqref{UC1} and \eqref{UC2}. These maps are used to construct the map $\alpha(\cdot)$ in \eqref{BG}.

\end{remark}

\begin{remark}
Kalman's article \cite{Kalman} does not provide additional properties for the map $\alpha(\cdot)$ from \eqref{BG}. However, as done in the work of B. Zhou \cite[Lemma 4]{Zhou-21} the condition \eqref{BG} with $\alpha(\cdot)\in \mathcal{B}$ is referred as the \textbf{Kalman's condition}. 
\end{remark}

Proposition \ref{Prop1} describes the strong linkage between the UCC and 
the Kalman's condition. In addition, inequality (\ref{BG}) states that any solution $t\mapsto \Phi_{A}(t,t_{0})x_{0}$ of the plant (\ref{lin})
passing through $x_{0}$ at $t=t_{0}$ verifies
\begin{displaymath}
|\Phi_{A}(t,t_{0})x_{0}|\leq \alpha(|t-t_{0}|)|x_{0}|,    
\end{displaymath}
that is, the growth of any solution of (\ref{lin}) over an interval $[t_{0},t]$ (or $[t,t_{0}]$) is dependent of the time elapsed between $t$ and $t_{0}$ but is independent of the initial time $t_{0}$ (or $t$).

The next result describes a set of equivalences with the Kalman's condition:
\begin{lemma}
\label{EKC}
Let $\Phi_{A}(t,s)$ be the transition matrix of the linear system \eqref{lin}. The fo\-llo\-wing properties are equivalent:
\begin{itemize}
\item[i)] There exist constants $K>1$ and $\beta>0$ such that:
\begin{equation}
\label{BG3}
\|\Phi_{A}(t,s)\|\leq Ke^{\beta|t-s|}\quad \textnormal{for any $t,s\in [0,+\infty)$}.
\end{equation}
\item[ii)]  The linear system \eqref{lin} satisfies the Kalman's condition.
\item[iii)] For any $h>0$ there exists $C_{h}>1$ such that any solution $t\mapsto x(t)$ of
\eqref{lin} verifies
\begin{equation}
\label{BG3-eq}
|x(t)|\leq C_{h}|x(s)| \quad \textnormal{for any $t\in [s-h,s+h]$}.
\end{equation} 
\end{itemize}
\end{lemma}

\begin{proof}
The proof of
i) $\Rightarrow$ ii) is direct since $u\mapsto \alpha(u)=Ke^{\alpha |u|}$ is continuous
and we can deduce $\alpha(\cdot) \in \mathcal{B}$, that is, $\alpha(\cdot)$ maps bounded sets into bounded sets.

Proof of ii) $\Rightarrow $ iii): let us consider an arbitrary constant $h>0$
and a solution $t\mapsto x(t)$ of (\ref{lin}). Note that (\ref{BG}) implies: 
$$
|\Phi_{A}(t,s)x(s)|=|x(t)|\leq \alpha(|t-s|)|x(s)|.
$$

Now, if  $t\in [s-h,s+h]$, which is equivalent to $h\geq |t-s|$, then the above inequality
implies that
\begin{displaymath}
|x(t)|\leq \sup\limits_{|\theta|\leq h}\alpha(|\theta|)|x(s)|,    
\end{displaymath}
and (\ref{BG3-eq}) is verified with $C_{h}=\max\left\{1+h,\sup\limits_{|\theta|\leq h}\alpha(|\theta|)\right\}$. Note that $C_{h}$ is well defined since $\alpha(\cdot)\in \mathcal{B}$.

Proof of iii) $\Rightarrow$ i): Given a fixed $h>0$ there exists $C_{h}>1$ such that the property iii)
is verified. Firstly, we will suppose that $t\geq s$, then we can assume the existence of $n\in \mathbb{N}$ such that
$t\in [s+(n-1)h,s+nh]$. By a recursive application of the property iii) we can
deduce that
$|x(t)|\leq C_{h}^{n}|x(s)|$. On the other hand, we also can see 
that $n-1\leq \frac{t-s}{h}\leq n$, which allow us to verify that:
$$
|x(t)|=|\Phi_{A}(t,s)x(s)|\leq C_{h}^{n}|x(s)|\leq C_{h}\,e^{\frac{\ln(C_{h})}{h}(t-s)}|x(s)|,
$$
and (\ref{BG3}) is verified with $K=C_{h}>1$ and $\beta=\frac{\ln(C_{h})}{h}>0$. The case $t\leq s$
can be proved in a similar way.
\end{proof}

The properties (i) and (iii) are well known in the qualitative theory of LTV systems. 
Currently are called as \textit{uniform bounded growth} and also allows the limit case $\beta=0$ and $C_{h}=1$. In addition, its equi\-valence has been proved by W. Coppel in \cite[pp. 8--9]{Coppel}. We refer to S. Siegmund \cite{Siegmund-2002} and K.J. Palmer \cite{Palmer} for more details.

As we stated before, the uniform bounded growth property is less known in control theory and, to the best of our knowledge, its equivalence with the Kalman's condition seems not have been noticed in the literature. 

The next result, whose proof has been sketched by Kalman in \cite[p.157]{Kalman}, des\-cribes a more surprising relation between the properties of uniform complete controllability and the Kalman's condition:
\begin{proposition}
\label{Prop2}
If any two of the properties \eqref{UC1},\eqref{UC2} and \eqref{BG} hold, the remaining one is also true.
\end{proposition}
The respective definitions of nonuniform complete controllability and nonuniform bounded growth -which will be detailed and explored in section 3- will allow us to obtain a result similar to the preceding proposition for the nonuniform situation.

\subsection{Feedback stabilizability}

By considering linear inputs of the form
\begin{equation}
\label{LFC}
    u(t)=-F(t)x(t),
\end{equation}
the control system (\ref{control1})  becomes a closed loop LTV system represented by 
\begin{equation}
\label{A-BF}
    \dot{x}(t)=\hat{A}(t)x(t) \quad \textnormal{with} \quad \hat{A}(t)=A(t)-B(t)F(t),
\end{equation}
whose corresponding transition matrix will be denoted by $\Phi_{\hat{A}}(\cdot,\cdot)$ or $\Phi_{A-BF}(\cdot,\cdot)$. The matrix function $F(\cdot)$ in (\ref{LFC}) is called \textit{feedback gain}. A far--reaching question in control theory is to determine the existence of a feedback gain such that the origin is an asymptotically stable solution of (\ref{A-BF}), i.e. an \emph{stabilizing feedback}.\\

For the autonomous framework, the controllability implies that the eigenvalues of the closed loop matrix $\hat A=A-BF$ can be arbitrarily assigned\footnote{Complex eigenvalues must come in conjugate pairs.} with an appropriate feedback matrix $F$. Contrarily to the autonomous case, there exists a myriad of asymptotic stabilities for nonautonomous systems and we refer the reader to \cite{Hahn} and \cite[Ch.VIII]{Rouche} for details. Particularly, a work from Ikeda \textit{et al.} \cite{Ikeda} describes the links between the CC and UCC properties of the control system (\ref{control1}) and the asymptotic stability of the closed loop system (\ref{A-BF}) summarized as follows:

\begin{proposition}\cite[Th.1]{Ikeda}
\label{IMK1}
The linear control system \eqref{control1} is completely controllable if and only if for any initial time
$t_{0}$ and any continuous and monotonically nondecreasing function $t\mapsto \delta(t,t_{0})$ defined
for any $t\geq t_{0}$ such that $\delta(t_{0},t_{0})=0$, there exist a feedback gain $F(t)$ defined for $t\geq t_{0}$ and a number $a(t_{0})>0$ such that
\begin{equation}
\label{A-BF_cs}
\|\Phi_{A-BF}(t,t_{0})\|\leq a(t_{0})e^{-\delta(t,t_{0})} \quad \textnormal{for all $t\geq t_{0}$}. 
\end{equation}
\end{proposition}

The property (\ref{A-BF_cs}) characterizes the \textit{complete stabilizability} of system (\ref{control1}) and $t\mapsto \delta(t,t_{0})$ is called \textit{measure of decay} in \cite{Ikeda}.
The CC is equivalent to the existence of a feedback gain $F(t)$ such that the closed loop system (\ref{A-BF}) is asymptotically stable for any measure of decay previously designed. For alternative approaches relating CC and feedback stabilizability, we refer to \cite[Th.3.2]{Anderson}.
\begin{proposition}\cite[Th.2]{Ikeda}
\label{IMK2}
If the linear control system \eqref{control1} is uniformly completely controllable, then for any $m>0$ there exists a feedback $F(t)$ defined for any $t\geq 0$ and a positive number $M$ such that
\begin{displaymath}
\|\Phi_{A-BF}(t,t_{0})\|\leq Me^{-m(t-t_{0})} \quad \textnormal{for all $t_{0}\geq 0$ and $t\geq t_{0}$}.    
\end{displaymath}
\end{proposition}

As we can see, the UCC implies the existence of a feedback gain $F(t)$ such that the closed loop
system (\ref{A-BF}) is uniformly exponentially stable for any rate of exponential decay previously chosen. 
A converse result is also proved in \cite[Th.3]{Ikeda} provided that
$A(\cdot)$ and $B(\cdot)$ are bounded matrices. For complementary approaches relating UCC and feedback stabilizability, we refer the reader to \cite[Prop. 4.3]{Ilchmann} and \cite[Th.4]{Zhou-21}.

It is important to emphasize that Propositions \ref{IMK1} and \ref{IMK2} are respectively asso\-cia\-ted to the
properties of \textit{complete stabilizability} \cite[Def.3]{Ikeda} and \textit{uniform complete stabilizability} \cite[Def.4]{Ikeda}. Nevertheless, there exist several related definitions of stabilizability and we refer the reader to \cite[Sec.2.3]{Anderson},\cite[Def.2.2]{Phat} and \cite[Def.2.8]{Rotea}.

\subsection{Exponential dichotomy and bounded growth properties}

The properties of exponential dichotomies are ubiquitous in the study of linear systems (\ref{lin}) but its use in linear control systems is less widespread. We will see that it provides us an interesting way to study
the relation between complete controllability and feedback stabilization. Let us recall its definition
for an arbitrary linear system
\begin{equation}
\label{ALS}
\dot{x}=V(t)x
\end{equation}
but having in mind the plant (\ref{lin}) and the closed loop system (\ref{A-BF}).

\begin{definition}
\label{Ed-A}
The system \eqref{ALS} has a uniform exponential dichotomy (UED) on $\mathbb{R}_{0}^{+}:=[0,+\infty)$ if there exist a projector 
$P(\cdot)$ and a couple of constants  $\mathcal{M}\geq 1$ and $\lambda >0$ such that
\begin{displaymath}
\left\{\begin{array}{rcll}
P(t)\Phi_{V}(t,t_{0})&=&\Phi_{V}(t,t_{0})P(t_{0})  & \textnormal{for all $t,t_{0}\geq 0$},\\
\|\Phi_{V}(t,t_{0})P(t_{0})\|&\leq & \mathcal{M}e^{-\lambda (t-t_{0})}  &  \textnormal{for all $t\geq t_{0}\geq 0$},\\
\|\Phi_{V}(t,t_{0})[I-P(t_{0})]\| &\leq & \mathcal{M}e^{-\lambda (t_{0}-t)} &  \textnormal{for all $t_{0}\geq t\geq 0$}.
\end{array}\right.
\end{displaymath}
\end{definition}

We refer the reader to \cite{Coppel,Lin2,Mitropolski} for additional details. Moreover, the above property can be seen as a special case of the nonuniform exponential dichotomy, defined as follows:

\begin{definition}
\label{NUEd-A}
The system \eqref{ALS} has a nonuniform exponential dichotomy (NUED) on $\mathbb{R}_{0}^{+}$ if there exists a projector 
$P(\cdot)$ together with three constants $\mathcal{M}\geq 1$, $\lambda >0$ and $\varepsilon \in [0,\lambda)$ such that
\begin{displaymath}
\left\{\begin{array}{rcll}
P(t)\Phi_{V}(t,t_{0})&=&\Phi_{V}(t,t_{0})P(t_{0})  &  \textnormal{for all $t,t_{0}\geq 0$},\\
\|\Phi_{V}(t,t_{0})P(t_{0})\|&\leq & \mathcal{M}e^{\varepsilon t_{0}}e^{-\lambda (t-t_{0})}  &  \textnormal{for all $t\geq t_{0}\geq 0$},\\
\|\Phi_{V}(t,t_{0})[I-P(t_{0})]\| &\leq & \mathcal{M}e^{\varepsilon t_{0}}e^{-\lambda (t_{0}-t)} &  \textnormal{for all $t_{0}\geq t\geq 0$}.
\end{array}\right.
\end{displaymath}
\end{definition} 

The exponential dichotomy mimics the hiperbolicity condition of the au\-to\-no\-mous case in the sense that
the projector splits the solutions of (\ref{ALS}) between contractive (stable) and expansive (unstable). Despite the apparent similarity between definitions \ref{Ed-A} and \ref{NUEd-A}, we stress that, in the uniform case, contractions as well as expansions depend only on the elapsed time between $t$ and $t_{0}$ while; in the nonuniform framework, they are also dependent of $t_{0}$. This motivates to denote $e^{\varepsilon t_{0}}$ 
as the \textit{nonuniformity}, which induces qualitative differences.

There are spectral theories constructed on the basis of these dichotomy properties. The key definitions
are the following:
\begin{definition}
\label{Espectro}
The uniform exponential dichotomy spectrum of the linear system \eqref{ALS} is defined by the set
\begin{displaymath}
\Sigma_{U}(V)=\left\{ \gamma \in \mathbb{R} \colon \dot{x}=[V(t)-\gamma I]x \quad \textnormal{has not a UED on $\mathbb{R}_{0}^{+}$}\right\}.    
\end{displaymath}
\end{definition}

\begin{definition}
 \label{NUED}
The nonuniform nonuniform exponential dichotomy spectrum 
of the linear system \eqref{ALS} is defined by the set
\begin{displaymath}
\Sigma_{NU}(V)=\left\{ \gamma \in \mathbb{R} \colon \dot{x}=[V(t)-\gamma I]x \quad \textnormal{has not a NUED on $\mathbb{R}_{0}^{+}$}\right\}.    
\end{displaymath}
\end{definition}

The spectrum $\Sigma_{U}(V)$ has been deeply studied and we refer to \cite{Klo,Siegmund-2002,Siegmund2} for
a detailed description and characterization. The spectrum $\Sigma_{NU}(V)$ has been studied in \cite{Chu,Zhang}, where the authors emulate the previous results and refer to \cite{Silva2023} and \cite{Zhu} for additional results.

The spectra $\Sigma_{U}(V)$ and $\Sigma_{NU}(V)$ have been characterized as the finite union of closed intervals, also called \textit{spectral intervals}, which mimic the role of the real part of the eigenvalues. These spectra are not necessarily bounded but it is well known that $\Sigma_{U}(V)$ is bounded when (\ref{ALS}) has the property of uniform bounded growth defined in (\ref{BG3}). Similarly, the spectrum $\Sigma_{NU}(V)$ is bounded when (\ref{ALS}) has the property of nonuniform bounded growth defined as follows:
\begin{definition} 
 \label{NUBGT}
The linear system \eqref{ALS} has a nonuniform bounded growth on the interval $J\subset [0,+\infty)$ if there exist constants $K_{0}>0$, $a>0$ and $\eta>0$ such that its transition matrix satisfies
\begin{equation}
\label{Phibound}
\|\Phi_{V}(t,\tau)\|\leq K_{0}e^{\eta \tau }e^{a|t-\tau|}\quad \textnormal{for any $t,\tau\in J$}.
\end{equation}
\end{definition}

We refer to \cite{Chu} and \cite{Zhou-JDE} for more details about the nonuniform bounded growth property. An example of this property has been considered in \cite[p.21]{Barreira}:
\begin{equation*}
\dot{x}=[\lambda_0+at\sin(t)]x \quad \textnormal{with $\lambda_{0}<a<0$}.
\end{equation*}

In fact, its transition matrix is given by
\begin{displaymath}
\Phi(t,s)=e^{\lambda_{0}(t-s)-a\cos(t)(t-s)-as(\cos(t)-\cos(s))+a(\sin(t)-\sin(s)))},
\end{displaymath}
and, it can be proved that
\begin{displaymath}
|\Phi(t,s)|\leq e^{2|a|s}e^{(|\lambda_{0}|+2|a|)|t-s|},    
\end{displaymath}
while the Kalman's condition (\ref{BG}) is not verified. Indeed, otherwise, by Lemma \ref{EKC} it would exist
constants $K\geq 1$ and $\alpha>0$ such that
$$
|\Phi(t,s)|\leq Ke^{\alpha|t-s|},
$$
that is, 
\begin{displaymath}
\lambda_{0}(t-s)-a\cos(t)(t-s)-as\{\cos(t)-\cos(s)\}+a(\sin(t)-\sin(s))) \leq  \ln(K)+\alpha |t-s|.
\end{displaymath}

The above expression is true when $t=s$. Now, the inequality 
\begin{displaymath}
\left\{\lambda_{0}-a\cos(t)\right\}\frac{(t-s)}{|t-s|}-as\frac{\{\cos(t)-\cos(s)\}}{|t-s|}+a\frac{(\sin(t)-\sin(s)))}{|t-s|} \leq  \frac{\ln(K)}{|t-s|}+\alpha,
\end{displaymath}
must be verified for any $t\neq s$, in particular, for any couple of sequences $\{t_{n}\}_{n}$ and $\{s_{n}\}_{n}$ such that $t_{n}\neq s_{n}$ for any $n\in \mathbb{N}$. Now, when considering the sequences $t_{n}=2n\pi$ and $s_{n}=2n\pi-\frac{\pi}{2}$ and evaluate at $t=t_{n}$ and $s=s_{n}$, it follows that $t_{n}-s_{n}=\frac{\pi}{2}$ and we have 
\begin{displaymath}
\lambda_{0}-a-a\left(2n\pi-\frac{\pi}{2}\right)\frac{2}{\pi}+a\frac{2}{\pi} \leq  2\frac{\ln(K)}{\pi}+\alpha \quad \textnormal{for any $n\in \mathbb{N}$},
\end{displaymath}
finally, as $a<0$, by letting $n\to +\infty$ we will obtain a contradiction.\\

\subsection{Nonuniform asymptotic stability properties}
Contrarily to the au\-to\-no\-mous case, there is not a univocal definition of asymptotic stability and there exist several ways to characterize them. We will focus on a limit case of the nonuniform exponential dichotomy studied previously. The formal definitions are:
\begin{definition}
\label{Nuni}
The closed loop system \eqref{A-BF} is:
\begin{itemize}
\item[i)]
$\tilde{a}(t_{0})$--nonuniformly exponentially stable if there exist constants $\tilde{a}(t_{0})>0$, $\mathcal{M}\geq 1$, $\lambda >0$ such that
\begin{equation}
\label{StaNonU-a}
\|\Phi_{A-BF}(t,t_{0})\|\leq \mathcal{M}\tilde{a}(t_{0})e^{-\lambda(t-t_{0})} \quad \textnormal{for all}\; t\geq t_{0}\geq 0.
\end{equation}
\item[ii)]
$e^{\varepsilon t_{0}}$--nonuniformly exponentially stable if there exist constants $\mathcal{M}\geq 1$, $\lambda >0$ and $\varepsilon \in [0,\lambda)$ such that
\begin{equation}
\label{StaNonU}
\|\Phi_{A-BF}(t,t_{0})\|\leq \mathcal{M}e^{\varepsilon t_{0}}e^{-\lambda(t-t_{0})} \quad \textnormal{for all}\; t\geq t_{0}\geq 0.
\end{equation}
\end{itemize}
\end{definition}

Note that the nonuniform exponential stability described by (\ref{StaNonU}) can be seen as a limit case of Definition \ref{NUEd-A} with $P(s)=I$. Similarly, when $\varepsilon=0$ in (\ref{StaNonU})
we obtain a limit case of Definition \ref{Ed-A} with $P(s)=I$ which coincides with the uniform exponential stability.

The numbers $\tilde{a}(t_{0})$ and $e^{\varepsilon t_{0}}$ are called \textit{nonuniformities}
and constitute the main difference between uniform and nonuniform exponential stabilities.
In fact, at is was pointed out in \cite[p.101]{Rugh}, for the uniform situation  the exponential decay is only dependent of the elapsed time $t-t_{0}$  while, in (\ref{StaNonU-a})--(\ref{StaNonU}), this decay is also dependent of the initial time $t_{0}$, considered by the nonuniformities. For additional details, we refer the reader to \cite{Zhou-16},\cite{Zhou-17} and \cite{Zhou-21}.

Note that the stabilization for the closed loop system (\ref{A-BF}) described by (\ref{StaNonU-a}) is a particular
case of the complete stabilizability described by Proposition \ref{IMK1} by considering
a specific measure of decay $\delta(t,t_{0})=\lambda (t-t_{0})$.

\begin{remark}
\label{dico}
Despite that the stability described by \eqref{StaNonU} is a particular case of \eqref{StaNonU-a}, we reserve a special treatment for it since corresponds to the purely contractive case of the nonuniform exponential dichotomy of Definition \ref{NUEd-A}. In this sense, given a $F(\cdot)$ such that the closed loop system  \eqref{A-BF} is $e^{\varepsilon t_{0}}$--nonuniformly exponentially stable, this is equivalent to the spectral property $$\Sigma_{NU}(A-BF)\subset (-\infty,0).$$
\end{remark}

\section{Controllability in a nonuniform
framework}

A strong consequence of the Propositions \ref{Prop1} and \ref{Prop2} stated in the previous section is that, if the plant of the linear control system (\ref{control1})
 does not verify the Kalman's condition (\ref{BG}), then the control system cannot be uniformly completely controllable but still can be completely controllable. In this context, and assuming that the property (\ref{BG}) is not verified, an open problem is:
 
\begin{itemize}
\item[a)] To explore the existence of a set of controllability properties halfway the complete controllability and uniform complete controllability.

\item[b)] Given a controllability property as in the above sense, to determine the asymptotic stability of the corresponding closed loop system, which must be stronger than the one stated in Proposition \ref{IMK1} but weaker than those stated in Proposition \ref{IMK2}.
\end{itemize}


There exist multiple ways to propose controllability properties having the above features. In Lemma \ref{EKC}, we proved that the classical Kalman's condition is actually the uniform bounded growth condition. We can re-interpret Kalman's Proposition \ref{Prop2} saying that inequalities \eqref{UC1}-\eqref{UC2}, along with the uniform bounded growth \eqref{BG3} form a trifecta. So, in order to move into the nonuniform situation, we will take the nonuniform bounded growth property \eqref{StaNonU} and introduce a \emph{nonuniform complete controllability} notion, also based on gramian inequalities, prompting to an extension of Propositions \ref{Prop1} and \ref{Prop2}.

\subsection{Nonuniform complete controllability: definition and consequences}



\begin{definition}
The linear control system \eqref{control1} is said to be \textbf{nonuniformly completely controllable}, if there exist fixed numbers $\mu_0\geq0$, $\mu_1\geq0$, $\tilde{\mu}_0\geq0$, $\tilde{\mu}_1\geq0$ and functions $\alpha_0(\cdot), \beta_{0}(\cdot),\alpha_1(\cdot),\beta_1(\cdot):[0,+\infty)\to(0,+\infty)$ such that for any $t\geq 0$, there exists $\sigma_0(t)>0$ with:
\begin{equation}
\label{alpha0alpha1}
    0<e^{-2\mu_0 t}\alpha_0(\sigma)I\leq W(t,t+\sigma)\leq e^{2\mu_1 t}\alpha_1(\sigma)I.
\end{equation}
\begin{equation}
\label{K}
    0<e^{-2\tilde{\mu}_0 t}\beta_0(\sigma) I\leq K(t,t+\sigma)\leq e^{2\tilde{\mu}_1 t}\beta_1(\sigma) I, 
\end{equation}
for every $\sigma\geq \sigma_{0}(t)$.
\end{definition}

\begin{remark}
    \label{Rmu}
In the context of the above property,
note that:
\begin{itemize}
\item[a)] If $\mu:=\max\{\mu_0, \tilde{\mu}_0, \mu_1, \tilde{\mu}_1\}$, the estimates \eqref{alpha0alpha1} and \eqref{K} can be stated in terms of $\mu$. Also, as done in the dichotomy $\&$ stability literature \cite{Barreira,Chu,Zhou-16}, the numbers $\mu_0$, $\tilde{\mu}_0$, $\mu_1$ and $\tilde{\mu}_1$ will be called the gramian nonuniformities.
\item[b)] When $\mu=0$ and $t\mapsto \sigma_0(t)$ is constant for all $t\geq0$, we recover the uniform complete controllability.  
\item[c)] By \eqref{alpha0alpha1} or \eqref{K}, it is straightforward to verify that NUCC implies CC.
\item [d)] By using \eqref{GP}, it follows directly that \eqref{K} is equivalent to 
\begin{equation}
    \label{beta0beta1}
 0<e^{-2\tilde{\mu}_0 t}\beta_0(\sigma) I\leq  \Phi_{A}(t+\sigma,t)W(t,t+\sigma)\Phi_{A}^{T}(t+\sigma,t) \leq e^{2\tilde{\mu}_1 t}\beta_1(\sigma) I.   
\end{equation}

\item[e)] Given an initial time $t\geq0$, there exists $\sigma_0(t)>0$ such that any state $x$ can be transferred to the origin in an interval $[t,t+\sigma_{0}(t)]$. Nevertheless, the energy remarks stated by Kalman are no longer valid since the bounds of $W(t,t+\sigma_0(t))$ are not uniform with respect to the initial time. We refer to Appendix \ref{app:energy} for additional details. 

\item[f)] In order to facilitate the writing, in some occasions when the context requires it, we will use the notation  $\sigma_{t}:=\sigma_0(t)$.
\end{itemize}
\end{remark}

If the linear control system (\ref{control1}) is nonuniformly completely controllable, we can see
that its corresponding plant admits properties and restrictions related to Propositions \ref{Prop1} y \ref{Prop2}, which 
will be described by the next results.

\begin{lemma}
\label{contasMT}
If the control system \eqref{control1} is nonuniformly completely controllable
with nonuniformities $\{\mu_0, \tilde{\mu}_0, \mu_1, \tilde{\mu}_1\}$ and positive functions $\{\alpha_{0},\beta_{0},\alpha_{1},\beta_{1}\}$, then for any $t\geq0$, there exists $\sigma_0(t)> 0$ such that for any $\sigma\geq\sigma_0(t)$, the plant of \eqref{control1} verifies the following estimates:
\begin{equation}
\label{bornesMT1}
e^{-(\tilde{\mu}_0+\mu_1) t}\sqrt{\frac{\beta_{0}(\sigma)}{\alpha_{1}(\sigma)}}\leq \|\Phi_{A}(t+\sigma,t)\|\leq     e^{(\mu_0+\tilde{\mu}_1) t}\sqrt{\frac{\beta_{1}(\sigma)}{\alpha_{0}(\sigma)}}.
\end{equation}
\end{lemma}
\begin{proof}
For any $x\neq0$, we can see that (\ref{alpha0alpha1}) is equivalent to
$$
0<e^{-2\mu_0 t}\alpha_0(\sigma)\|x\|^{2}\leq x^{T}W(t,t+\sigma)x\leq e^{2\mu_1 t}\alpha_1(\sigma)\|x\|^{2}.
$$

By considering the change of variable $x=\Phi_{A}^{T}(t+\sigma,t)\eta$,
with $\eta\neq0$, the previous estimate is transformed into
$$
e^{-2\mu_0 t}\alpha_0(\sigma)|\Phi_{A}^{T}(t+\sigma,t)\eta|^{2}\leq \eta^{T}K(t,t+\sigma)\eta\leq e^{2\mu_1 t}\alpha_1(\sigma)|\Phi_{A}^{T}(t+\sigma,t)\eta|^{2}. 
$$

On the other hand, as (\ref{K}) is equivalent to 
\begin{displaymath}
    0<e^{-2\tilde{\mu}_0 t}\beta_0(\sigma) |\eta|^{2}\leq \eta^{T}K(t,t+\sigma)\eta\leq e^{2\tilde{\mu}_1 t}\beta_1(\sigma) |\eta|^{2},
\end{displaymath}
the above two inequalities leads to the following estimates:
$$
e^{-2\tilde{\mu}_0 t}\beta_{0}(\sigma)|\eta|^{2}\leq e^{2\mu_1 t}\alpha_{1}(\sigma)|\Phi_A^{T}(t+\sigma,t)\eta|^{2},
$$
and
$$
e^{-2\mu_0 t}\alpha_{0}(\sigma)|\Phi_{A}^{T}(t+\sigma,t)\eta|^{2} \leq e^{2\tilde{\mu}_1 t}\beta_{1}(\sigma)|\eta|^{2},
$$
which implies that
$$
e^{-2(\tilde{\mu}_0+\mu_1) t}\frac{\beta_{0}(\sigma)}{\alpha_{1}(\sigma)} \leq \left( \frac{|\Phi_{A}^{T}(t+\sigma,t)\eta|}{|\eta|}\right)^{2}  \leq e^{2(\mu_0+\tilde{\mu}_1) t}\frac{\beta_{1}(\sigma)}{\alpha_{0}(\sigma)}
$$
and the result follows by considering the supreme over $\eta\neq0$ and recalling that $\|\Phi_{A}\|=\|\Phi_{A}^{T}\|$.
\end{proof}

A direct consequence from inequality (\ref{bornesMT1}) is the estimation:
\begin{equation}
\label{bornesMT2}
e^{-(\mu_0+\tilde{\mu}_1) t}\sqrt{\frac{\alpha_{0}(\sigma)}{\beta_{1}(\sigma)}}\leq \|\Phi_{A}(t+\sigma,t)\|^{-1}
\leq e^{(\tilde{\mu}_0+\mu_1) t}\sqrt{\frac{\alpha_{1}(\sigma)}{\beta_{0}(\sigma)}},
\end{equation}
which will be useful to prove the next result:
\begin{lemma}
\label{cotas2MT}
If the control system \eqref{control1} is nonuniformly completely controllable
with nonuniformities $\{\mu_0, \tilde{\mu}_0, \mu_1, \tilde{\mu}_1\}$ and positive functions $\{\alpha_{0},\beta_{0},\alpha_{1},\beta_{1}\}$, then for any $t\geq0$, there exists $\sigma_0(t)>0$ such that for any $\sigma\geq\sigma_0(t)$, the plant of \eqref{control1} verifies the following estimates:
\begin{equation}
\label{inversophi}
    e^{-(\mu_0+\tilde{\mu}_1)t}\sqrt{\frac{\alpha_0(\sigma)}{\beta_1(\sigma)}}\leq\|\Phi_{A}(t,t+\sigma)\|\leq  e^{(\tilde{\mu}_0+\mu_{1})t}\sqrt{\frac{\alpha_1(\sigma)}{\beta_0(\sigma)}}.
\end{equation}
\end{lemma}
\begin{proof}
Since $\|I\|=1$, we have that \begin{displaymath}
1=\|\Phi_{A}(t,t+\sigma)\Phi_{A}(t+\sigma,t)\|  \leq \|\Phi_{A}(t,t+\sigma)\|\, \|\Phi_{A}(t+\sigma,t)\|,
\end{displaymath}
and, by using (\ref{bornesMT2}) we conclude that
\begin{equation*}
 e^{-(\mu_0+\tilde{\mu}_1) t}\sqrt{\frac{\alpha_0(\sigma)}{\beta_1(\sigma)}}\leq \|\Phi_{A}(t+\sigma,t)\|^{-1} \leq \|\Phi_{A}(t,t+\sigma)\|.
\end{equation*}
In addition, the inequality  (\ref{beta0beta1}) is equivalent to
\begin{displaymath}
e^{-2\tilde{\mu}_0 t}\beta_0(\sigma)|x|^{2} \leq x^{T}\Phi_{A}(t+\sigma,t)W(t,t+\sigma)\Phi_{A}^{T}(t+\sigma,t)x\leq e^{2\tilde{\mu}_1 t}\beta_1(\sigma) |x|^{2}
\end{displaymath}
for any $x\neq 0$. Therefore, by choosing $\eta=\Phi_{A}^{T}(t+\sigma,t)x$, this estimation becomes:
\begin{displaymath}
 \begin{array}{l}
e^{-2\tilde{\mu}_0 t}\beta_0(\sigma)|\Phi_{A}^{T}(t,t+\sigma)\eta|^{2} \leq \eta^{T}W(t,t+\sigma)\eta\leq e^{2\tilde{\mu}_1 t}\beta_{1}(\sigma)|\Phi_{A}^{T}(t,t+\sigma)\eta|^{2}.
\end{array}
\end{displaymath}

By combining the left side of the previous inequality with the right side of (\ref{alpha0alpha1}), we obtain that
\begin{displaymath}
    e^{-2\tilde{\mu}_0 t}\beta_0(\sigma)|\Phi_{A}^{T}(t,t+\sigma)\eta|^{2} \leq \eta^{T}W(t,t+\sigma)\eta\leq e^{2\mu_1 t}\alpha_{1}(\sigma)|\eta|^{2},
\end{displaymath}
which implies that 
\begin{displaymath}
\left(\frac{|\Phi_{A}^{T}(t,t+\sigma)\eta|}{|\eta|}\right)^{2}\leq e^{2(\tilde{\mu}_0+\mu_1) t}\frac{\alpha_{1}(\sigma)}{\beta_{0}(\sigma)}.    
\end{displaymath}
Finally, by taking the supreme over any $\eta\neq 0$, we deduce that 
$$
\|\Phi_{A}(t,t+\sigma)\|\leq e^{(\tilde{\mu}_0+\mu_1) t}\sqrt{ \frac{\alpha_{1}(\sigma)}{\beta_{0}(\sigma)}},
$$
and the proof follows.
\end{proof}

\begin{corollary}
For any $\eta \neq 0$ it follows that for all $\sigma\geq\sigma_0(t)$
\begin{equation*}
e^{-(\mu_{0}+\tilde{\mu}_{1})t}\sqrt{\frac{\alpha_{0}(\sigma)}{\beta_{1}(\sigma)}}|\eta| \leq |\Phi_{A}^{T}(t,t+\sigma)\eta|,
\end{equation*}
and
\begin{equation}
\label{Cor-L2-a}
e^{-(\tilde{\mu}_{0}+\mu_{1})t}\sqrt{\frac{\beta_{0}(\sigma)}{\alpha_{1}(\sigma)}}|\eta| \leq |\Phi_{A}^{T}(t+\sigma,t)\eta|.
\end{equation}
\end{corollary}

\begin{proof}
We will prove the second inequality since the first one can be proved in a similar way. Now, by using the identity $|\eta|=|\Phi_{A}^{T}(t,t+\sigma)\Phi_{A}^{T}(t+\sigma,t)\eta|$, we can deduce that
\begin{displaymath}
\begin{array}{rcl}
|\eta|&\leq& \|\Phi_{A}^{T}(t,t+\sigma)\|\,|\Phi_{A}^{T}(t+\sigma,t)\eta| \\\\
&\leq& \|\Phi_{A}(t,t+\sigma)\|\,|\Phi_{A}^{T}(t+\sigma,t)\eta|,
\end{array}
\end{displaymath}
and also that
\begin{displaymath}
\|\Phi_{A}(t,t+\sigma)\|^{-1}|\eta|\leq |\Phi_{A}^{T}(t+\sigma,t)\eta|, 
\end{displaymath}
and (\ref{Cor-L2-a}) is a direct consequence from (\ref{inversophi}). 
\end{proof}

As we have seen, the preceding results deduce a set of estimates from the Gramian inequalities (\ref{alpha0alpha1})--(\ref{K}) and its nonuniformities. These results will allow us to prove that the inequalities imply a condition that generalizes the Kalman property (\ref{BG}).

\begin{lemma}
\label{NUIMBU}
If the control system \eqref{control1} is nonuniformly completely controllable, then there exist $\nu>0$ and a function  $\alpha(\cdot)\in\mathcal{B}$ satisfying
\begin{equation}
\label{crec-acot}
\|\Phi_{A}(t,\tau)\|\leq e^{\nu\tau}\,\alpha(|t-\tau|)  \quad \textnormal{for any $t,\tau\geq0$}.
\end{equation}
\end{lemma}
\begin{proof}
By considering that the system \eqref{control1} is nonuniformly completely controllable, given $t\geq0$, $\tau\geq0$, there exist $\sigma_{t}, \sigma_{\tau}>0$ such that for any $\sigma\geq\max\{\sigma_{t}, \sigma_{\tau}\}$ Lemma \ref{contasMT} allows us to ensure that 
\begin{equation}
\label{phi1}
\|\Phi_{A}(\tau+\sigma,\tau)\|\leq e^{(\tilde{\mu}_1+\mu_0)\tau}\sqrt{\frac{\beta_{1}(\sigma)}{\alpha_{0}(\sigma)}}
\end{equation}
and Lemma \ref{cotas2MT} leads to that 
\begin{equation}
\label{phi2}
\|\Phi_{A}(t,t+\sigma)\|\leq e^{(\mu_1+\tilde{\mu}_0)t}\sqrt{\frac{\alpha_{1}(\sigma)}{\beta_{0}(\sigma)}}.
\end{equation}

If we consider $|t-\tau|\geq\sigma_0:=\max\{\sigma_{t}, \sigma_{\tau}\}$, we have to keep in mind the following subcases:

\noindent Case A.1: $t>\tau$. We choose $\sigma=t-\tau\geq\sigma_0$, then we have that for $t=\tau+\sigma$, the inequality (\ref{phi1}) can be expressed as:
$$
\|\Phi_{A}(t,\tau)\|\leq e^{(\tilde{\mu}_1+\mu_0)\tau} \sqrt{\frac{\beta_{1}(t-\tau)}{\alpha_{0}(t-\tau)}}.
$$
\noindent Case A.2: $\tau>t$. We choose $\sigma=\tau-t\geq\sigma_0$. Again, by considering $\tau=t+\sigma$, the inequality (\ref{phi2}) is given by:
$$
\|\Phi_{A}(t,\tau)\|\leq e^{(\mu_1+\tilde{\mu}_0)t}    \sqrt{\frac{\alpha_{1}(\tau-t)}{\beta_{0}(\tau-t)}}\leq e^{(\mu_1+\tilde{\mu}_0)\tau}    \sqrt{\frac{\alpha_{1}(\tau-t)}{\beta_{0}(\tau-t)}}.
$$

Gathering the above two cases, we take $\nu_1=\max\{\mu_0+\tilde{\mu}_1,\tilde{\mu}_0+\mu_1\}$ and deduce that
$$
\|\Phi_{A}(t,\tau)\| \leq e^{\nu_1\tau}\alpha_2(|t-\tau|)$$
where 
$$\alpha_2(|t-\tau|):=\max\left\{\sqrt{\frac{\beta_{1}(|t-\tau|)}{\alpha_{0}(|t-\tau|)}}, \sqrt{\frac{\alpha_{1}(|t-\tau|)}{\beta_{0}(|t-\tau|)}}\right\}.
$$


On the other hand, we will consider the case when $|t-\tau|<\sigma_0:=\max\{\sigma_{t},\sigma_{\tau}\}$.

\noindent Case B.1: $\tau<t<\tau+\sigma_{0}\leq t+\sigma_0$. From Lemma \ref{cotas2MT} we have easily that for any $\sigma\geq\sigma_0$ we have that:
$$
\begin{array}{rcl}
\|\Phi_{A}(t,\tau)\| &\leq&  \displaystyle\|\Phi_{A}(t,t+\sigma)\| \,\|\Phi_{A}(t+\sigma,\tau)\|, \\\\
&\leq &
 \displaystyle e^{(\mu_1+\tilde{\mu}_0)t} \sqrt{\frac{\alpha_{1}(\sigma)}{\beta_{0}(\sigma)}}\|\Phi_{A}(t+\sigma,\tau)\|.
\end{array}
$$

Moreover, as $t+\sigma-\tau>\sigma\geq\sigma_{0}$, the estimation obtained for the case A.1, combined with $t<\tau+\sigma_0\leq\tau+\sigma$, allow us to deduce
that
$$
\begin{array}{rcl}
\|\Phi_{A}(t,\tau)\| &\leq&  \displaystyle e^{(\mu_1+\tilde{\mu}_0)t} \sqrt{\frac{\alpha_{1}(\sigma)}{\beta_{0}(\sigma)}} e^{(\tilde{\mu}_1+\mu_0)\tau}\sqrt{\frac{\beta_1(t+\sigma-\tau)}{\alpha_0(t+\sigma-\tau)}},\\\\
&\leq& \displaystyle e^{(\mu_1+\tilde{\mu}_0)(\tau+\sigma)}\sqrt{\frac{\alpha_{1}(\sigma)\beta_1(t+\sigma-\tau)}{\beta_{0}(\sigma)\alpha_0(t+\sigma-\tau)}} e^{(\tilde{\mu}_1+\mu_0)\tau}.


\end{array}
$$


\noindent Case B.2: $t<\tau<t+\sigma_0<\tau+\sigma_0$. Similarly as the previous case, by using 
Lemma \ref{contasMT} and the estimation of the case A.2 arising from $\tau+\sigma-t>\sigma\geq\sigma_{0}$, then for any $\sigma\geq\sigma_0$
we obtain that:
$$
\begin{array}{rcl}
\|\Phi_{A}(t,\tau)\| &\leq&  \displaystyle\|\Phi_{A}(t,\tau+\sigma)\| \,\|\Phi_{A}(\tau+\sigma,\tau)\|, \\\\
&\leq& \displaystyle\|\Phi_{A}(t,\tau+\sigma)\| e^{(\mu_0+\tilde{\mu}_1)\tau}\sqrt{\frac{\beta_{1}(\sigma)}{\alpha_{0}(\sigma)}},\\\\
&\leq&\displaystyle e^{(\mu_1+\tilde{\mu}_0)t}\sqrt{\frac{\alpha_1(\tau+\sigma-t)}{\beta_0(\tau+\sigma-t)}}e^{(\mu_0+\tilde{\mu}_1)\tau}\sqrt{\frac{\beta_{1}(\sigma)}{\alpha_{0}(\sigma)}},\\\\
&\leq&\displaystyle e^{(\mu_1+\tilde{\mu}_0)(\tau+\sigma)}\sqrt{\frac{\alpha_1(\tau+\sigma-t)\beta_1(\sigma)}{\beta_0(\tau+\sigma-t)\alpha_0(\sigma)}}e^{(\tilde{\mu}_1+\mu_0)\tau}.
\end{array}
$$

Based on the last two cases, by choosing $\nu_2=\mu_1+\tilde{\mu}_0+\mu_0+\tilde{\mu}_1$, we have that
$$\|\Phi_{A}(t,\tau)\|\leq e^{\nu_2\tau}\alpha_3(|t-\tau|),$$
where
$$\alpha_3(|t-\tau|)=\displaystyle e^{(\mu_1+\tilde{\mu}_0)\sigma}\max\left \{\sqrt{\frac{\alpha_{1}(\sigma)\beta_1(|t-\tau|+\sigma)}{\beta_{0}(\sigma)\alpha_0(|t-\tau|+\sigma)}} , \sqrt{\frac{\alpha_{1}(|\tau-t|+\sigma)\beta_1(\sigma)}{\beta_{0}(|\tau-t|+\sigma)\alpha_0(\sigma)}}\right \}.$$

Finally, due to the fact that $\nu_1\leq\nu_2$, then 
$$\|\Phi_A(t,\tau)\|\leq e^{\nu \tau}\alpha(|t-\tau|),$$
where  
$$\nu=\nu_2\quad \textnormal{and}\quad \alpha(|t-\tau|)=\max\{\alpha_2(|t-\tau|), \alpha_3(|t-\tau|)\}.$$

\end{proof}

\begin{remark}
From now on, the property \eqref{crec-acot} will be called as the \textit{nonuniform Kalman's property}.
\end{remark}

\begin{remark}
\label{NUBKP}
We can see that the nonuniform bounded growth from Definition \ref{NUBGT} is an example of the nonuniform Kalman's condition \eqref{crec-acot}.
Nevertheless, the equivalences stated in Lemma \ref{EKC} for the uniform case cannot be directly adapted
to the nonuniform context. More specifically, the work with intervals developed in the implicance 
$iii) \Rightarrow i)$ cannot be generalized due to the nonuniformities. 
\end{remark}

The above lemmas will allow us to generalize the Kalman's Proposition \ref{Prop2} to the nonuniform framework, as follows:
\begin{theorem}
\label{p2l3}
Any two of the properties \eqref{alpha0alpha1}, \eqref{K} and \eqref{crec-acot}
imply the third one.
\end{theorem}

\begin{proof}
Firstly, a direct consequence of Lemma \ref{NUIMBU} is that (\ref{alpha0alpha1}) and (\ref{K})
imply the property (\ref{crec-acot}).

Secondly, let us assume that the properties (\ref{alpha0alpha1}) and (\ref{crec-acot}) are satisfied. In other words, for any $t\geq0$, there exists $\sigma_0(t)>0$ such that for all $\sigma\geq\sigma_0(t)$, the inequality (\ref{alpha0alpha1}) is satisfied.
By recalling that
$$
\|\Phi_{A}(t,t+\sigma)\|^{-1}\leq\|\Phi_{A}(t+\sigma,t)\|
$$
and by using (\ref{crec-acot}), we obtain that
\begin{equation}
\label{Phiarribaabajo}
\displaystyle \frac{1}{e^{\nu t}}\frac{1}{e^{\nu \sigma}\alpha(\sigma)}\leq\|\Phi_{A}(t+\sigma,t)\|\leq e^{\nu t}\alpha(\sigma).
\end{equation}

By considering $x=\Phi_{A}^{T}(t+\sigma,t)\eta$ with $\eta\neq 0$, the inequalities (\ref{alpha0alpha1}) combined with the identity (\ref{GP}) imply
\begin{displaymath}
 e^{-2\mu_{0}t}\alpha_{0}(\sigma)|\Phi_{A}^{T}(t+\sigma,t)\eta|^{2}\leq \eta^{T}K(t,t+\sigma)\eta \leq e^{2\mu_{1}t}\alpha_{1}(\sigma)|\Phi_{A}^{T}(t+\sigma,t)\eta|^{2}.
\end{displaymath}
    
Now, by considering the inequalities seen in Corollary 1:
\begin{displaymath}
\|\Phi_{A}^{T}(t,t+\sigma)\|^{-1}|\eta|\leq |\Phi_{A}^{T}(t+\sigma,t)\eta|\leq \|\Phi_{A}^{T}(t+\sigma,t)\| |\eta|,    
\end{displaymath}
combined with (\ref{Phiarribaabajo}) and taking supreme over $\eta\neq 0$, we obtain that
$$e^{-2\nu t}\frac{1}{e^{2\nu\sigma}\alpha^{2}(\sigma)}e^{-2\mu_0 t}\alpha_0(\sigma)I\leq K(t,t+\sigma)\leq e^{2\nu t}\alpha^{2}(\sigma)e^{2\mu_1 t}\alpha_1(\sigma)I$$
and the condition (\ref{K}) is satisfied for all $\sigma\geq\sigma_0(t)$.

Finally, and following the same line of proof in the previous case, let us assume that the conditions (\ref{K}) and (\ref{crec-acot}) are satisfied. We use (\ref{crec-acot}) in order to obtain a similar expression to (\ref{Phiarribaabajo}):
    $$\frac{1}{e^{\nu t}}\frac{1}{\alpha(\sigma)}\leq\|\Phi_{A}(t,t+\sigma)\|\leq e^{\nu t}e^{\nu \sigma}\alpha(\sigma)$$
    and by combining it with (\ref{K}) we have that:
    $$e^{-2\nu t}\frac{1}{\alpha^{2}(\sigma)}e^{-2\tilde{\mu}_0 t}\beta_0(\sigma)<W(t,t+\sigma)<e^{2\nu t}e^{2\nu\sigma}\alpha^{2}(\sigma)e^{2\tilde{\mu}_1 t}\beta_1(\sigma)$$
    then the estimate (\ref{alpha0alpha1}) follows.
\end{proof}

A noticeable consequence of Theorem \ref{p2l3} is that any control system whose plant has the nonuniform Kalman's property (\ref{crec-acot}) cannot be UCC when $\nu>0$, but could be NUCC provided that (\ref{alpha0alpha1}) and/or (\ref{K}) are satisfied. In this case, it is desirable to know examples of systems having the nonuniform Kalman's property and the nonuniform bounded growth described by (\ref{Phibound})
provides a distinguished one, as it was stated in Remark \ref{NUBKP}. Note that we have arrived to a point where the controllability analysis of the 1960's unexpectedly meets the qualitative theory of linear nonautonomous systems.

The new concept of NUCC is more restrictive than the classical UCC. In the subsection \ref{ex:NUCC} it is presented an example of a NUCC system, while in the subsection \ref{ex:CC-notNUCC} it is showed that this new concept is effectively an intermediate stage between UCC and CC. We remark that the example exhibited in the subsection \ref{ex:CC-notNUCC} was introduced by Kalman himself in his seminal work \cite[p.157]{Kalman} as a CC system that does not fulfill the UCC conditions.

\subsection{An example of a nonuniformly completely controllable system}\label{ex:NUCC} If we assume that the linear system (\ref{lin}) admits the property of nonuniform bounded growth with parameters $(K_0,a,\eta)$, it is interesting to explore some conditions on $B(t)$ leading to the estimate (\ref{alpha0alpha1}) and, as a consequence of the Theorem \ref{p2l3}, leading to a NUCC linear control system.

In fact, we will assume that for any $t\geq0$: 
    $$b_0e^{-\beta_0 t}I\leq B(t)B^{T}(t)\leq b_1e^{\beta_1 t}I.$$
    Note that the previous condition includes the particular, and ubiquitous, case of constant $B$, with $BB^T>0$. 
    Based on the transition matrix property, we have that for any vector $x\neq0$, it follows that  
$
x=\Phi_{A}^{T}(s,t)\Phi_{A}^{T}(t,s)x
$
and then:    
    $$|x|^{2}\leq\left\|\Phi_{A}^{T}(s,t)\right\|^{2}|\Phi_{A}^{T}(t,s)x|^{2}\Rightarrow \frac{|x|^{2}}{\left\|\Phi_{A}(s,t)\right\|^{2}}\leq|\Phi_{A}(t,s)x|^{2}, $$
    and by considering the nonuniform bounded growth hypothesis, we can ensure that
    \begin{equation*}
    \left\|\Phi_{A}(s,t)\right\|\leq K_0e^{a|s-t|+\eta t}\Leftrightarrow\frac{1}{K_0e^{a|s-t|+\eta t}}\leq \frac{1}{\left\|\Phi_{A}(s,t)\right\|}.
    \end{equation*}

    From the above, we will obtain the right-hand estimate of (\ref{alpha0alpha1}). For any $t\geq0$, there exist $\sigma_0(t)>0$ such that for any $\sigma\geq\sigma_0(t)$ and $x\neq0$:
    \begin{displaymath}
    \begin{array}{rcl}
    \displaystyle\int_{t}^{t+\sigma}x^{T}\Phi_{A}(t,s)B(s)B^{T}(s)\Phi_{A}^{T}(t,s)x\;ds&\leq& \displaystyle\int_{t}^{t+\sigma}K_0^{2}e^{2a(s-t)+2\eta s}b_1^{2}e^{2\beta_1 s}|x|^{2}\;ds\\\\
    &=&\displaystyle K_0^{2}b_1^{2}|x|^{2}e^{-2at}\int_{t}^{t+\sigma}e^{2(a+\eta+\beta_1)s}\;ds,
    \end{array}
    \end{displaymath}
which implies that
 \begin{displaymath}
    \begin{array}{rcl}
    x^{T}W(t,t+\sigma)x&\leq& 
    \displaystyle \frac{K_0^{2}b_1^{2}}{2(a+\eta+\beta_1)}e^{-2at}\left [e^{2(a+\eta+\beta_1)(t+\sigma)}-e^{2(a+\eta+\beta_1)t}\right]|x|^{2}\\\\
    &=&\displaystyle \frac{K_0^{2}b_1^{2}}{2(a+\eta+\beta_1)}\left [e^{2(a+\eta+\beta_1)\sigma}-1\right]e^{2(\eta+\beta_1)t}|x|^{2}.\\\\
    \end{array}
    \end{displaymath}

Similarly, notice that
    $$\begin{array}{rcl}
         \displaystyle\int_{t}^{t+\sigma}x^{T}\Phi_{A}(t,s)B(s)B^{T}(s)\Phi_{A}^{T}(t,s)x\;ds &\geq&  \displaystyle\int_{t}^{t+\sigma}\frac{1}{K_0^{2}}e^{-2a(s-t)-2\eta t}b_0^{2}e^{-2\beta_0 s}|x|^{2}\;ds\\\\
         &=&\displaystyle\frac{b_0^{2}}{K_0^{2}}|x|^{2}e^{2(a-\eta)t}\int_{t}^{t+\sigma}e^{-2(a+\beta_0)s}\;ds,
    \end{array}$$
which implies
\begin{displaymath}
    \begin{array}{rcl}
         x^{T}W(t,t+\sigma)x &\geq&  \displaystyle\frac{b_0^{2}}{K_0^{2}2(a+\beta_0)}e^{2(a-\eta)t}\left[e^{-2(a+\beta_0)t}-e^{-2(a+\beta_0)(t+\sigma)}\right]|x|^{2}\\\\
         &=&\displaystyle \frac{b_0^{2}}{K_0^{2}2(a+\beta_0)}\left[1-e^{-2(a+\beta_0)\sigma}\right]e^{-2(\eta+\beta_0)t}|x|^{2}
    \end{array}
\end{displaymath}
and the nonuniform complete controllability follows.

As an extra comment, in the previous proof we can notice that $\sigma_0(t)=\sigma_0>0$ is constant and works in order to prove the inequalities for all $t\geq0$.\\

\subsection{An example of a CC but not NUCC system}\label{ex:CC-notNUCC}

In \cite[p.157]{Kalman}, Kalman proved that the LTV control system
\begin{equation}
\label{ex-BV}
\dot{x}=-tx +\sqrt{2(t-1)}e^{-t+1/2}u(t)\quad \textnormal{with $t\geq 1$}.
\end{equation}
is CC but not UCC. 

In fact, it is easy to check that the transition matrix
$\Phi(t,\tau)=e^{\frac{\tau^{2}-t^{2}}{2}}$ does not verify the property of
uniform bounded growth (\ref{BG3}) and, by Lemma \ref{EKC}, the
Kalman condition (\ref{BG}) cannot be verified. Now,
it follows from Proposition \ref{Prop1} that (\ref{ex-BV}) does not have the property of 
uniform complete controllability.

On the other hand, Kalman also verifies that
$$
W(t,t+\sigma)=e^{2(\sigma-1)t+(\sigma-1)^{2}}-e^{-2t+1},
$$
If $\sigma=1$, we will have that
$$
W(t,t+1)=1-e^{-2t+1}>0 \quad \textnormal{since $t\geq 1$},
$$
and for any $t_{0}\geq 1$ we can choose $t_{f}=t_{0}+1$ such that
$$
W(t_{0},t_{f})=1-e^{-2t_{0}+1}>0,
$$
and it follows that (\ref{ex-BV}) is completely controllable.

Now, we will verify that (\ref{ex-BV}) cannot be NUCC. Indeed, otherwise, 
the Gramian inequalities (\ref{alpha0alpha1})--(\ref{K}) will be verified
and Lemma \ref{NUIMBU} would imply the existence of $\alpha(\cdot)\in\mathcal{B}$ 
and $\nu>0$ such that
$$
\Phi(t,\tau)=e^{\frac{\tau^{2}-t^{2}}{2}}\leq \alpha(|t-\tau|)e^{\nu \tau},
$$
which is equivalent to
$$
(\tau-t)(\tau+t)=\tau^{2}-t^{2} \leq 2\ln\left(\alpha(|t-\tau|)\right)+2\nu \tau.
$$



Now, given the above constant $\nu$, let us consider $\tau=\beta_{n}$  and $t=\beta_{n}-3\nu$ 
where the sequence $\{\beta_n\}$ verifies $1+3\nu<2\beta_{n}$ for any $n\in \mathbb{N}$ and $\beta_{n}\to +\infty$. Then, the above inequality is equivalent to:
$$
3\nu(2\beta_{n}-3\nu)\leq 2\ln\left(\alpha(3\nu)\right)+2\nu \beta_{n} \quad \textnormal{or} \quad 3\nu \leq \frac{2\ln(\alpha(3\nu))}{2\beta_{n}-3\nu}+\frac{2\nu \beta_{n}}{2\beta_{n}-3\nu}.
$$

Now, by letting $n\to +\infty$, and using the fact that $\nu>0$, the above inequality leads to a contradiction.

\section{Feedback stabilization for nonuniformly completely controllable
control systems}


\subsection{Nonuniform exponential stabilities and the nonuniform exponential dichotomy}
The main result of this section states that a family of nonuniformly completely controllable linear control systems (\ref{control1})
having the property of nonuniform bounded growth can be stabilized by a feedback gain $F(t)$ such that the closed loop system (\ref{A-BF}) is nonuniformly exponentially stable. Firstly, let us consider the family of Riccati differential equations parametrized by the number $\mathcal{L}>0$: 
\begin{equation}
\label{ER}
\dot{S}(t)+\left(A(t)+\mathcal{L} I\right)^{T} S(t)+S(t)\left(A(t)+ \mathcal{L} I\right)-S(t) B(t) B^{T}(t) S(t)=-I, 
\end{equation}
where $A(\cdot)$ and $B(\cdot)$ are the matrices of the linear control system (\ref{control1}). This
family of equations will play an essential role in the present section, whose main result is:
\begin{theorem}
\label{T2}
Assume that the linear control system \eqref{control1} is nonuniformly completely controllable and that the following additional properties are verified:
\begin{itemize}
\item[\textbf{(H1)}] The corresponding plant has the property of nonuniform bounded growth on $[0,+\infty)$ with
constants $(K_{0},a,\eta)$,
\item[\textbf{(H2)}] The property \eqref{alpha0alpha1} is verified with constants $\mu_{0}>0$ and $\mu_{1}>0$.
\end{itemize}

In addition, if $t\mapsto S_{\mathcal{L}}(t)$ is a solution of the Riccati equation \eqref{ER}
with 
\begin{equation}
\label{L-NU}
\mathcal{L}>2(\theta_{2}+2\theta_{1}), \quad \textit{where $\theta_{1}=\mu_{1}+4\eta$ and $\theta_{2}=\eta+2(\mu_{1}+\mu_{0})$},
\end{equation}
then, for any rate of decay $\mathcal{L}-\theta_{1}>2\theta_{2}+3\theta_{1}$ previously chosen, there exist a constant $M\geq 1$ and a
feedback gain $F(t)$ such that the 
closed loop system \eqref{A-BF} is $e^{(\theta_{1}+\theta_{2})t_{0}}$--nonuniformly exponentially stable, that is, 
\begin{equation}
\label{NUES}
\|\Phi_{A-BF}(t,t_{0})\|\leq Me^{(\theta_{1}+\theta_{2})t_{0}} e^{-(\mathcal{L}-\theta_{1})(t-t_{0})}   \quad \textnormal{for any $t\geq t_{0}$}. 
\end{equation}
\end{theorem}

\medskip

The above result deserves several comments: 

\noindent 1.- As usual in feedback stabilization results, $F(t)$ is constructed by using a solution of the Riccati equation (\ref{ER}), whose existence is ensured in the next subsection.

\medskip

\noindent 2.- Note that (\ref{NUES}) coincides with (\ref{StaNonU}) when considering
$\lambda=\mathcal{L}-\theta_{1}>0$ and $\varepsilon=\theta_{1}+\theta_{2}$. We also stress that (\ref{L-NU}) implies 
$\varepsilon \in [0,\lambda)$.

\medskip

\noindent 3.- 
As the NUCC is a particular case of CC, we can compare Theorem \ref{T2} with Proposition \ref{IMK1} (\cite[Th.1]{Ikeda}). In fact, if we
consider the measure decay $\delta(t,t_{0})=(\mathcal{L}-\theta_{1})(t-t_{0})$ such that (\ref{L-NU}) is verified, we will have the existence of $a(t_{0})>0$ such that
\begin{equation}
\label{NUES2}
\|\Phi_{A-BF}(t,t_{0})\|\leq a(t_{0})e^{-(\mathcal{L}-\theta_{1})(t-t_{0})}   \quad \textnormal{for any $t\geq t_{0}$}, 
\end{equation}

Despite the formal similarity between (\ref{NUES}) and (\ref{NUES2}), we emphasize that a meticulous reading of \cite[p.724]{Ikeda} shows that the constant $a(t_{0})$
from Proposition \ref{IMK1} is given explicitly by $a(t_{0})=[1+\lambda_{\max}(S_{\mathcal{L}}(t_{0}))]^{1/2}$, where $S_{\mathcal{L}}$ is solution of the Riccati equation (\ref{ER}). In turn, from the expression \eqref{NUES} we can deduce that $a(t_0)=Me^{(\vartheta_1+\vartheta_2)t_0}$. 

\medskip

\noindent  4.- In Proposition \ref{IMK2} (\cite[Th.2]{Ikeda}), corresponding to a UCC system, the exponential decay of the closed loop system can be arbitrarily chosen. In Theorem \ref{T2}, the exponential decay $\mathcal{L}-\theta_{1}$ can be chosen only for $\mathcal{L}>2(\theta_{2}+2\theta_{1})$, where $\theta_1$ and $\theta_2$ depend on the available bounds $\mu_0$ and $\mu_1$ for the gramian $W$ and the number $\eta$ of the bounded growth property. An attentive reading of our proof will show that this restriction is due to a technical result, namely, a sufficient condition of nonuniform exponential stability (Proposition \ref{T3}), and still remains a challenge to overcome.

\medskip

\noindent 5.- Remark \ref{dico} allows us to see Theorem \ref{T2} from the perspective of Definition \ref{NUED}: there exists a feedback gain 
$F(t)$ such that the close loop system (\ref{A-BF}) has a nonuniform exponential dichotomy on $[0,+\infty)$ with the identity as projector and constants $M\geq 1$, $\mathcal{L}-\theta_{1}>0$ and $\theta_{1}+\theta_{2} \in [0,\mathcal{L}-\theta_{1})$.\\

Additionally, Theorem \ref{T2} can be interpreted from a spectral point of view relating NUCC with
the localization of $\Sigma_{NU}$ as far to the left as desired:
\begin{corollary}\label{coro-spectrum}
If a linear control system \eqref{control1} is NUCC and verifies \textbf{(H1)--(H2)},
then for any $\mathcal{L}\in (2(\theta_{2}+2\theta_{1}),+\infty)$, there exists a feedback gain $F(t)$ such that 
the nonuniform spectrum of \eqref{A-BF} verifies 
$$
\Sigma_{NU}(A-BF)\subset (-\infty,-\mathcal{L}+\theta_{2}+2\theta_{1})\subset (-\infty,0).
$$
\end{corollary}

\begin{proof}
We will prove an equivalent contention, namely
$$
[-\mathcal{L}+\theta_{2}+2\theta_{1},+\infty)\subset \left[\Sigma_{NU}(A-BF)\right]^{c}.
$$

Firstly, notice that if $\lambda \geq -\mathcal{L}+(\theta_{2}+2\theta_{1})$ it follows that
\begin{equation}
\label{specstab}
\lambda+\mathcal{L}-\theta_{1}>\theta_{1}+\theta_{2}>0.
\end{equation}

Secondly, by Theorem \ref{T2} we know that (\ref{NUES}) is verified and we multiply it
by $e^{-\lambda(t-t_{0})}$ with $t\geq t_{0}$ and $\lambda\geq -\mathcal{L}+(\theta_{2}+2\theta_{1})$,
which leads to
$$
\|e^{-\lambda(t-t_0)}\Phi_{A-BF}(t,t_0)\|\leq Me^{(\theta_1+\theta_2)t_0}e^{-(\lambda+\mathcal{L}-\theta_1)(t-t_0)} \quad \textnormal{for any $t\geq t_{0}\geq 0$}.
$$

Notice that, as $\Phi_{A-BF-\lambda I}(t,s)=\Phi_{A-BF}(t,s)e^{-\lambda(t-s)}$, the above estimation
can be written as follows:
$$
\|\Phi_{A-BF-\lambda I}(t,t_0)\|\leq Me^{(\theta_1+\theta_2)t_0}e^{-(\lambda+\mathcal{L}-\theta_1)(t-t_0)}
\quad \textnormal{for any $t\geq t_{0}\geq 0$}.
$$

Finally, the above inequality combined with (\ref{specstab}) implies that
the system $$\dot{x}=(A(t)-B(t)F(t)-\lambda I)x$$
admits nonuniform exponential dichotomy with identity projector. Then we have that  
$\lambda \in \left[\Sigma_{NU}(A-BF)\right]^{c}$, which concludes the proof.
\end{proof}



\subsection{Proof of Theorem \ref{T2}} 
The proof will be made in several steps.

\medskip

\noindent \textit{Step 1: Auxiliary results}. Riccati equations are ubiquitous in feedback stabilization results and, in this case, the equation (\ref{ER})
will be related to the following result:
\begin{proposition} 
\label{T3}
If the linear system
\begin{equation}
\label{Piloto}
\dot{z}=U(t)z
\end{equation}
verifies the following properties:

\begin{itemize}
\item[i)] There exists a positive definite operator $S(t)\in M_{n}(\mathbb{R})$ of class $C^1$ in $t>0$ and constants $C_1>0$, $C_2>0$, $\phi_1\geq0$ and $\phi_2\geq0$ such that
\begin{equation}
\label{CotaparaS}
C_1e^{-2\phi_1 t}I\leq S(t) \leq C_2 e^{2\phi_2 t}I, 
\end{equation}
\item[ii)] There exists a constant $\mathcal{L}>2(\phi_2+2\phi_1)$ such that
\begin{equation}
\label{RiccatiS}
\dot{S}(t)+U^{T}(t) S(t)+S(t) U(t) \leq-(\mathrm{Id}+\mathcal{L} S(t)),
\end{equation}
\end{itemize}
then the system \eqref{Piloto} is $e^{(\phi_{1}+\phi_{2})t_{0}}$--nonuniformly exponentially stable, namely, there exist constants $M\geq 1$, $\lambda=\frac{\mathcal{L}}{2}-\phi_{1}>0$
and $\varepsilon=\phi_{1}+\phi_{2}\in (0,\lambda)$ such that
\begin{equation*}
||\Phi_{U}(t,t_{0})||\leq  Me^{(\phi_1+\phi_2)t_{0}}e^{-(\frac{\mathcal{L}}{2}-\phi_1)(t-t_{0})}  \quad \textnormal{for any $t\geq t_{0} \geq 0$}.
\end{equation*}

\end{proposition}
\begin{proof}
Firstly, let us construct the map $H:(0,+\infty) \times \R^{n} \rightarrow [0,+\infty)$ defined by
\begin{equation}
\label{DefH}
H(t, x)=\langle S(t) x, x\rangle,
\end{equation}
where $S(t)$ is the  positive definite operator stated in i).  


Let $t\mapsto x(t)$ be a solution of (\ref{Piloto}), by using (\ref{CotaparaS}) we easily
deduce that 
\begin{equation}
\label{inegA}
 C_1e^{-2\phi_1 t}|x(t)|^2\leq H(t,x(t))\leq C_{2}e^{2\phi_{2}t}|x(t)|.
\end{equation}

Secondly, note that for $t \geq \tau\geq 0$, we have that
\begin{equation}
\label{Prosof}
H(t, x(t)) \leq e^{-\mathcal{L}(t-\tau)} H(\tau, x(\tau)) .
\end{equation}

In fact, by (\ref{DefH}) combined with (\ref{RiccatiS}), it is straightforward to deduce that
\begin{displaymath}
\begin{array}{rcl}
\frac{d}{dt}H(t,x(t))&=&x^{T}(t)\{S'(t)+S(t)U(t)+U^{T}(t)S(t)\}x(t),\\
  &\leq &  -|x(t)|^{2}-\mathcal{L}x^{T}(t)S(t)x(t), \\
  &\leq & - \mathcal{L}H(t,x(t))
\end{array}
\end{displaymath}
and (\ref{Prosof}) follows by the comparison lemma for scalar differential equations.

Finally, by considering (\ref{inegA}) and (\ref{Prosof}), we obtain that for any $t\geq \tau$:
$$
\begin{array}{rcl}
|\Phi_{U}(t, \tau) x(\tau)|^2 & =&|x(t)|^2
 \leq \displaystyle\frac{1}{C_1}e^{2\phi_1 t} H(t, x(t)) \leq \frac{1}{C_1}e^{2\phi_1 t}e^{-\mathcal{L}(t-\tau)} H(\tau, x(\tau)), \\
& \leq& \displaystyle \frac{1}{C_1} e^{2\phi_1 t}e^{-\mathcal{L}(t-\tau)}C_2 e^{2\phi_2 \tau}|x(\tau)|^2,\\
& \leq& \displaystyle\frac{C_2}{C_1}e^{-(\mathcal{L}-2\phi_1)(t-\tau)+2(\phi_1+\phi_2) \tau}|x(\tau)|^2,
\end{array}
$$
and therefore, if we define $M=\max\left \{1,\sqrt{\frac{C_2}{C_1}}\right \}$, we have that
$$
\|\Phi_{U}(t, \tau)\| \leq Me^{(\phi_{1}+\phi_{2})\tau}e^{-(\frac{\mathcal{L}}{2}-\phi_1)(t-\tau)}.
$$

Now, notice that $\mathcal{L}>2(2\phi_1+\phi_2)$ implies $\lambda=\frac{\mathcal{L}}{2}-\phi_{1}>\phi_{1}+\phi_{2}=\varepsilon>0$ and the result follows.
\end{proof}

The above result is inspired in Theorem 2.2 from \cite{Liao}, which provides a
necessary condition ensuring that (\ref{Piloto}) is a nonuniform contraction, namely, a type of nonuniform asymptotic stability that is more general than those stated in Definition \ref{Nuni}. The fact of working with a more specific stability allows us to consider a less restrictive condition as 
(\ref{CotaparaS}).



\begin{lemma}
\label{L1}
    If the linear control system \eqref{control1} is nonuniformly completely controllable, then the perturbed control system
\begin{equation}
\label{LTV-l}
\dot{y}(t)=[A(t)+\ell I]y(t)+B(t)u(t),
\end{equation}
with $\ell>0$, is also nonuniformly completely controllable.
\end{lemma}    

\begin{proof}
As we said before, the transition matrix associated to the plant
of (\ref{LTV-l}) is given by $\Phi_{A+\ell I}(t,s)=\Phi_{A}(t,s)e^{\ell(t-s)}$, then the corresponding controllability gramian matrix  of (\ref{LTV-l}) is
given by:
\begin{equation*}
\begin{array}{rcl}
W_{\ell}(t,t+\sigma)&=&\displaystyle\int_{t}^{t+\sigma}\Phi_{A+\ell I}(t,s)B(s)B^{T}(s)\Phi_{A+\ell I}^{T}(t,s)\;ds, \\\\
&=&\displaystyle\int_{t}^{t+\sigma}e^{2\ell(t-s)}\Phi_{A}(t,s)B(s)B^{T}(s)\Phi_{A}^{T}(t,s)\;ds.
\end{array}
\end{equation*}

By considering the nonuniform complete controllability of the system (\ref{control1}), then for any $t\geq0$, there exists $\sigma_0(t)>0$ such that for any $\sigma\geq\sigma_0(t)$ the estimate (\ref{alpha0alpha1}) is satisfied, then by combining this fact with the inequality $e^{-\ell \sigma}\leq e^{\ell(t-s)}\leq 1$ for any $t\leq s\leq t+\sigma$, we can deduce that
\begin{equation}
\label{EsGramPer}
0<\alpha_0(\sigma)e^{-2\ell\sigma}e^{-2\mu_0 t}I\leq W_{\ell}(t,t+\sigma)\leq\alpha_1(\sigma)e^{2\mu_1 t}I.
\end{equation}

By using again $\Phi_{A+\ell I}(t,s)=\Phi(t,s)e^{\ell(t-s)}$ combined with (\ref{crec-acot}) and $\ell>0$, it is straightforward to see that the plant of (\ref{LTV-l}) has the nonuniform Kalman's condition. Finally, Theorem \ref{p2l3} implies that (\ref{K}) is satisfied for $K_{\ell}(t,t+\sigma)$ and the Lemma follows.
\end{proof}

\begin{remark}
\label{exist-ricc}
As we pointed out in Remark \ref{Rmu}, the nonuniform complete controllability of \eqref{LTV-l} implies  its complete controllability, then by the Lemma \ref{L1} with $\ell=\mathcal{L}$ combined with Proposition 6.6 from \cite[p.159]{Kalman} and the assumptions \textbf{(H1)--(H2)}, the existence of solutions $S_{\mathcal{L}}(\cdot)$ of the Riccati equation \eqref{ER} is ensured\footnote{By Proposition 6.6 from \cite{Kalman} we know that
$S_{\mathcal{L}}(t)=\lim\limits_{t_{1}\to +\infty}\Pi(t,0,t_{1})$ where $t\mapsto \Pi(t,0,t_{1})$ is solution of (\ref{ER}) for any $t\leq t_{1}$ and verifies the 
terminal condition $\Pi(t_{1},0,t_{1})=0$.}. 
\end{remark}

Notice that if we consider a linear feedback input $u(t)=-F(t)y(t)$, the control system (\ref{LTV-l}) with the particular choice $\ell=\frac{\mathcal{L}}{2}$ becomes:
\begin{equation}
\label{A+LI-BF}
\dot{y}(t)=\left[A(t)+\frac{\mathcal{L}}{2} I\right]y(t)+B(t)u(t)=\left[A(t)+\frac{\mathcal{L}}{2} I-B(t)F(t)\right]y(t).
\end{equation}

The Proposition \ref{T3} will be used to deduce sufficient conditions ensuring 
the nonuniform exponential stability for
the linear system (\ref{A+LI-BF}). In order to address this task, we need to introduce a second result, which is an adaptation from Ikeda {\it et al.}  \cite[Lemma 4]{Ikeda} to the nonuniform context:

\begin{lemma}
\label{evolutionoperatorresult}
If the linear control system \eqref{control1} 
satisfies the hypothesis \textbf{(H1)} 
then there exist functions $\gamma_1(\cdot)$ and $\gamma_2(\cdot)$ such that for any $\sigma>0$ we have that:
\begin{equation}
\label{integralPhi}
\gamma_1(\sigma)e^{-2\eta t}I\leq\displaystyle\int_{t}^{t+\sigma}\Phi_{A}^{T}(s,t)\Phi_{A}(s,t)\;ds\leq \gamma_2(\sigma)e^{2\eta t}I.
\end{equation}
\end{lemma}
\begin{proof}
By \textbf{(H1)} we know that the plant has the property of nonuniform bounded growth with 
constants $(K_0,a,\eta)$. Then, by using (\ref{Phibound}) we can write 
    $$
    \displaystyle
    \frac{1}{K_0^{2}}e^{-2a|t-s|-2\eta s}|x|^{2}\leq |\Phi_{A}(s,t)x|^{2}\leq K_0^{2}e^{2a|s-t|+2\eta t}|x|^{2} \quad \textnormal{for any $x\neq 0$},
$$ 
where the left estimation is obtained by noticing that $|x|=|\Phi_{A}(t,s)\Phi_{A}(s,t)x|$ combined
with (\ref{Phibound}). Now, we can deduce
$$\displaystyle
    \int_{t}^{t+\sigma}\frac{1}{K_0^{2}}e^{-2a(s-t)-2\eta s}|x|^{2}\;ds\leq\displaystyle\int_{t}^{t+\sigma}|\Phi_{A}(s,t)x|^{2}\;ds\leq\displaystyle\int_{t}^{t+\sigma} K_0^{2}e^{2a(s-t)+2\eta t}|x|^{2}\;ds.
    $$

In relation to the left hand inequality, we can ensure that
$$
    \begin{array}{rcl}
    \displaystyle\int_{t}^{t+\sigma} \frac{1}{K_0^{2}}e^{-2a(s-t)-2\eta s}|x|^{2}\;ds&=&\displaystyle\frac{1}{K_0^{2}2(a+\eta)}\left [e^{-2\eta t}-e^{-2a\sigma-2\eta t}\right ]|x|^2,\\\\
    &=&\displaystyle\frac{1}{K_0^{2}2(a+\eta)}e^{-2\eta t}\left [1-e^{-2a\sigma}\right ]|x|^{2}
    \end{array}
    $$
and about the right expression we have that
$$
    \begin{array}{rcl}
    \displaystyle\int_{t}^{t+\sigma} K_0^{2}e^{2a(s-t)+2\eta t}|x|^{2}\;ds&=&\displaystyle\frac{K_0^{2}}{2a}\left [e^{2a\sigma+2\eta t}-e^{2\eta t}\right ]|x|^2,\\\\
    &=&\displaystyle\frac{K_0^{2}}{2a}e^{2\eta t}\left [e^{2a\sigma}-1\right ]|x|^{2}
    \end{array}
    $$
and (\ref{integralPhi}) follows by considering $$\gamma_1(\sigma)=\frac{1}{K_0^{2}2(a+\eta)}\left [1-e^{-2a\sigma}\right]\quad \textnormal{and}\quad\displaystyle\gamma_2(\sigma)=\frac{K_0^{2}}{2a}\left[e^{2a\sigma}-1\right ].$$
\end{proof}

\noindent\textit{Step 2: Feedback stabilization of the linear control system (\ref{LTV-l})}

\begin{lemma}
Under the assumptions of Theorem \ref{T2}, for any $\mathcal{L}>2(\theta_{2}+2\theta_{1})$ there exist $M\geq 1$
and a linear feedback $u(t)=-F(t)x(t)$ for such that the shifted control system (\ref{A+LI-BF}) is
$e^{(\theta_{1}+\theta_{2})s}$--nonuniformly exponentially stable, namely,
\begin{equation}
\label{est-sof}
||\Phi_{A+\frac{\mathcal{L}}{2}I-BF}(t,s)||\leq Me^{-(\frac{\mathcal{L}}{2}-\theta_1)(t-s)+(\theta_1+\theta_2) s} \quad \textnormal{for all}\; t\geq s\geq 0.
\end{equation}

\end{lemma}

\begin{proof}
As it was previously pointed out in Remark \ref{exist-ricc}, the existence of a solution $S_{\mathcal{L}}(\cdot)$ for the Riccati equation (\ref{ER}) is ensured. Now, let us consider the linear system (\ref{A+LI-BF}) with a feedback gain defined by:
\begin{equation}
\label{FL}
F(t)=\frac{1}{2} B^{T}(t) S_{\mathcal{L}}(t).
\end{equation}

Upon inserting (\ref{FL}) into (\ref{A+LI-BF}) we obtain 
\begin{equation}
\label{perturbadoBBS}
\dot{y}=\left[A(t)+\frac{\mathcal{L}}{2}I-B(t)F(t)\right]y=\left[A(t)+\frac{\mathcal{L}}{2}I-\frac{1}{2}B(t)B^{T}(t)S_{\mathcal{L}}(t)\right]y.
\end{equation}

We will prove that the linear system
(\ref{perturbadoBBS}) satisfies all the assumptions of Proposition \ref{T3}, with $U(t)=\hat{A}(t)+\frac{\mathcal{L}}{2}I$,
where $\hat{A}(t)=A(t)-B(t)F(t)$.


Firstly, we will verify that the solution $S_{L}(t)$ of the Riccati equation (\ref{ER}) satisfies the condition (\ref{CotaparaS}). In fact, as the 
nonuniform complete controllability implies complete controllability, by using Lemma \ref{L1} with $\ell=\frac{\mathcal{L}}{2}$, the Lemma 3 from \cite{Ikeda} also states that:
\begin{equation*}
D^{-1}(t)\leq S_\mathcal{L}(t)\leq E(t)
\end{equation*}
where the matrices $D(t)$ and $E(t)$ are defined by
$$
D(t)=Y_{\frac{\mathcal{L}}{2}}^{-1}\left(t, t_{c}(t)\right)+\operatorname{tr}\left(W_{\frac{\mathcal{L}}{2}}\left(t, t_{c}(t)\right)\right)\left(1+\frac{\operatorname{tr}\left(Y_{\frac{\mathcal{L}}{2}}\left(t, t_{c}(t)\right)\right)}{\lambda_{\min }\left(Y_{\frac{\mathcal{L}}{2}}\left(t, t_{c}(t)\right)\right)}\right)^{2} I
$$
and
$$
E(t)=W_{\frac{\mathcal{L}}{2}}^{-1}\left(t, t_{c}(t)\right)+\operatorname{tr}\left(Y_{\frac{\mathcal{L}}{2}}\left(t, t_{c}(t)\right)\right)\left(1+\frac{\operatorname{tr}\left(W_{\frac{\mathcal{L}}{2}}\left(t, t_{c}(t)\right)\right)}{\lambda_{\min }\left(W_{\frac{\mathcal{L}}{2}}\left(t, t_{c}(t)\right)\right)}\right)^{2} I,
$$
with $Y_{\frac{\mathcal{L}}{2}}(t,t_{c}(t))$ described by
\begin{equation*}
    Y_{\frac{\mathcal{L}}{2}}(t, t_{c}(t))=\displaystyle\int_{t}^{t_{c}(t)}\Phi_{A+\frac{\mathcal{L}}{2}I}^{T}(s,t)\Phi_{A+\frac{\mathcal{L}}{2}I}(s,t)\;ds,
\end{equation*}
where $t_{c}(t)$ is any number such that $W_{\frac{\mathcal{L}}{2}}(t,t_{c}(t))$ is positive definite. 

A direct consequence of Lemma \ref{L1} with $\ell=\frac{\mathcal{L}}{2}$ is that for all $t\geq0$, there exists $\sigma_0(t)>0$, the gramian matrix $W_{\frac{\mathcal{L}}{2}}(t,t+\sigma)$
is positive definite for any $\sigma\geq\sigma_0(t)$. In consequence, we will consider $t_{c}(t)=t+\sigma$.

In order to estimate the lower and upper bounds for $D^{-1}(t)$ and $E(t)$ respectively, we need to obtain 
auxiliary estimations. By using (\ref{EsGramPer}) with $\ell=\frac{\mathcal{L}}{2}$, we have that
\begin{displaymath}
    0<\alpha_0(\sigma)e^{-\mathcal{L}\sigma}e^{-2\mu_0 t}I\leq W_{\frac{\mathcal{L}}{2}}(t,t+\sigma)\leq\alpha_1(\sigma)e^{2\mu_1 t}I.
\end{displaymath}

In addition, by following the lines of Lemma \ref{evolutionoperatorresult}, we deduce the inequalities
\begin{displaymath}
\gamma_{1}^{\frac{\mathcal{L}}{2}}(\sigma)e^{-2\eta t}I\leq Y_{\frac{\mathcal{L}}{2}}(t,t+\sigma)\leq  \gamma_{2}^{\frac{\mathcal{L}}{2}}(\sigma)e^{2\eta t}I,
\end{displaymath}
where
$$
\displaystyle\gamma_1^{\frac{\mathcal{L}}{2}}(\sigma)=\frac{1}{K_0^{2}2(a+\frac{\mathcal{L}}{2}+\eta)}\left[1-e^{-2(a+\frac{\mathcal{L}}{2})\sigma}\right ]\quad \textnormal{and}\quad \gamma_2^{\frac{\mathcal{L}}{2}}(\sigma)=\frac{K_0^{2}}{2(a+\frac{\mathcal{L}}{2})}\left[e^{2(a+\frac{\mathcal{L}}{2})\sigma}-1\right ].
$$

The above estimations for $W_{\frac{\mathcal{L}}{2}}(\cdot,\cdot)$ and $Y_{\frac{\mathcal{L}}{2}}(\cdot,\cdot)$
combined with (\ref{TRVP}) allow us to deduce the following bounds:

\begin{displaymath}
\begin{tabular}{|p{\dimexpr 6.5cm-2\tabcolsep}|p{\dimexpr 6.5cm-2\tabcolsep}|}
\hline\vspace{.2cm} 
 \centering\ensuremath{Y_{\frac{\mathcal{L}}{2}}^{-1}(t,t+\sigma)\leq \frac{e^{2\eta t}}{\gamma_{1}^{\frac{\mathcal{L}}{2}}(\sigma)}I}    &\vspace{.2cm}
 \begin{center}\ensuremath{W_{\frac{\mathcal{L}}{2}}^{-1}(t,t+\sigma)\leq \frac{e^{\mathcal{L}\sigma+2\mu_{0}t}}{\alpha_{0}(\sigma)}}\end{center}
 
 \\
 \hline \vspace{.2cm}
  \centering\ensuremath{\operatorname{tr} Y_{\frac{\mathcal{L}}{2}}(t,t+\sigma)\leq n\gamma_{2}^{\frac{\mathcal{L}}{2}}(\sigma)e^{2\eta t}}  & \vspace{.2cm}
\begin{center}\ensuremath{\operatorname{tr}\,W_{\frac{\mathcal{L}}{2}}(t,t+\sigma)\leq n\alpha_1(\sigma)e^{2\mu_1 t}}\end{center}

\\
 \hline \vspace{.2cm}
 \centering\ensuremath{\gamma_{1}^{\frac{\mathcal{L}}{2}}(\sigma)e^{-2\eta t}\leq 
\lambda_{\min}(Y_{\frac{\mathcal{L}}{2}}(t,t+\sigma))}   & \vspace{.2cm}
\begin{center}\ensuremath{\alpha_0(\sigma)e^{-(\mathcal{L}\sigma+2\mu_0 t)} \leq
\lambda_{\min}(W_{\frac{\mathcal{L}}{2}}(t,t+\sigma))}\end{center}

\\
\hline
\end{tabular}
\end{displaymath}

By gathering the above left inequalities and using that $\theta_{1}=\mu_1+4\eta$, we can conclude that:
\begin{displaymath}
\begin{array}{rcl}
D(t)&\leq & \displaystyle\left[\frac{e^{2\eta t}}{\gamma_{1}^{\frac{\mathcal{L}}{2}}(\sigma)}+ n\alpha_1(\sigma)e^{2\mu_1 t}\left(1+\frac{n\gamma_{2}^{\frac{\mathcal{L}}{2}}(\sigma)e^{4\eta t}}{\gamma_{1}^{\frac{\mathcal{L}}{2}}(\sigma)}\right)^2\right]I \\
&\leq &  \displaystyle\max\left\{\frac{1}{\gamma_1^{\frac{\mathcal{L}}{2}}(\sigma)},n\alpha_1(\sigma)\right\}\left[e^{2\eta t}+e^{2\mu_1 t}\left(1+\frac{n\gamma_{2}^{\frac{\mathcal{L}}{2}}(\sigma)e^{4\eta t}}{\gamma_1^{\frac{\mathcal{L}}{2}}(\sigma)}\right)^{2}\right]I\\
&\leq &  \displaystyle\max\left\{\frac{1}{\gamma_1^{\frac{\mathcal{L}}{2}}(\sigma)},n\alpha_1(\sigma)\right\}\left[ e^{2\theta_{1} t}+e^{2\theta_1 t}\left(1+\frac{n\gamma_{2}^{\frac{\mathcal{L}}{2}}(\sigma)e^{\theta_{1} t}}{\gamma_1^{\frac{\mathcal{L}}{2}}(\sigma)}\right)^{2}\right]I\\\\
&\leq & \displaystyle\tilde{N}^{-1}e^{2\theta_{1}t}I,
\end{array}
\end{displaymath}
and, as $D(t)$ is positive definite, we can deduce that:
\begin{equation*}
\tilde{N}e^{-2\theta_1 t} I\leq D^{-1}(s)\leq S_{\mathcal{L}}(t).
\end{equation*}

Similarly, by using the above right inequalities and recalling that $\theta_2=\eta+2(\mu_1+\mu_0)$,
we have that:
\begin{displaymath}
\begin{array}{rcl}
S_{\mathcal{L}}(t)&\leq & \displaystyle E(t)\\
&\leq & \displaystyle \left[\frac{e^{\mathcal{L}\sigma+2\mu_0 t}}{\alpha_{0}(\sigma)}+n \gamma_{2}^{\frac{\mathcal{L}}{2}}(\sigma)e^{2\eta t}\left(1+\frac{n \alpha_{1}(\sigma)e^{2\mu_1 t}}{\alpha_{0}(\sigma)e^{-\mathcal{L}\sigma-2\mu_0 t}}\right)^{2}\right] I \\
&\leq &  \displaystyle \max\left\{\frac{e^{\mathcal{L}\sigma}}{\alpha_0(\sigma)},n\gamma_{2}^{\frac{\mathcal{L}}{2}}(\sigma) \right\}\left[e^{2\mu_0 t}+e^{2\eta t} \left(1+\frac{n \alpha_{1}(\sigma)e^{\mathcal{L}\sigma}e^{2(\mu_1+\mu_0) t}}{\alpha_{0}(\sigma)}\right)^{2}\right] I\\
&\leq &  \displaystyle \max\left\{\frac{e^{\mathcal{L}\sigma}}{\alpha_0(\sigma)},n\gamma_{2}^{\frac{\mathcal{L}}{2}}(\sigma) \right\}\left[e^{2\theta_{2} t}+e^{2\theta_{2} t} \left(1+\frac{n \alpha_{1}(\sigma)e^{\mathcal{L}\sigma}e^{2\theta_{2} t}}{\alpha_{0}(\sigma)}\right)^{2}\right] I,
\end{array}
\end{displaymath}
and we deduce that
\begin{equation*}
S_{\mathcal{L}}(t) \leq \tilde{M}e^{2\theta_2 t
} I,
\end{equation*}
and the estimate (\ref{CotaparaS}) is verified.

Finally, we will verify that the solution $S_{\mathcal{L}}(t)$ of the Riccati equation (\ref{ER}) satisfies the condition (\ref{RiccatiS}) from 
Proposition \ref{T3}. In order to do that, if we consider (\ref{A-BF}) and (\ref{FL}), we can see that $S_{\mathcal{L}}(t)$ can be seen as a solution of the following equation:
\begin{displaymath}
\begin{array}{c}
\dot{S}(t)+\left(A(t)+\frac{\mathcal{L}}{2} I\right)^{T} S(t)+S(t)\left(A(t)+\frac{\mathcal{L}}{2} I\right)-S(t) B(t) B^{T}(t) S(t)=-(I+\mathcal{L}S(t)).
\end{array}
\end{displaymath}

Under some minor transformations and considering $\hat{A}(t)=A(t)-B(t)F(t)$ with $F(t)$ defined by (\ref{FL}), the previous equation can be rewritten as follows:
\begin{equation*}
\begin{array}{lcl}
\dot{S}(t)+\left(\hat{A}(t)+\frac{\mathcal{L}}{2} I\right)^{T} S(t)+S(t)\left(\hat{A}(t)+\frac{\mathcal{L}}{2} I\right)&=&-(I+\mathcal{L}S(t)),
\end{array}
\end{equation*}
and the condition (\ref{RiccatiS}) is verified for any $\mathcal{L}>0$, where $U(t)=\hat{A}(t)+\frac{\mathcal{L}}{2}I$.

 Therefore by choosing $\mathcal{L}>2(\theta_2+2\theta_1)$, the Proposition  \ref{T3} ensures that the li\-near system (\ref{perturbadoBBS})
is nonuniformly exponentially stable, and more specifically, the inequality (\ref{est-sof}) is verified.
\end{proof}

\medskip

\noindent\textit{Step 3: End of proof.}
By using the identity $\Phi_{A-BF+\frac{\mathcal{L}}{2}I}(t,s)=\Phi_{A-BF}(t,s)e^{\frac{\mathcal{L}}{2}(t-s)}$ 
combined with (\ref{est-sof}), it follows that
\begin{displaymath}
\|\Phi_{A-BF}(t,s)\|\leq Me^{(\theta_{1}+\theta_{2})s}e^{-(\mathcal{L}-\theta_1)(t-s)} \quad \textnormal{for all $t\geq s\geq0$},
\end{displaymath}
and the Theorem follows.

\section{Conclusions and comments}


This paper introduced a new property of controllability for nonautonomous linear control systems: the nonuniform complete controllability (NUCC), which is more general than the uniform complete controllability (UCC) but is more specific than the complete controllability (CC), both classical in control theory. First, we reviewed the uniform context, established by R. Kalman, and proved that  the currently called \emph{Kalman's condition} is equivalent to the uniform bounded growth property (Lemma \ref{EKC}). Then, we moved into the nonuniform context, introduced the NUCC property and proved that this new notion of controlability is strongly related to the nonuniform version of Kalman condition (Theorem \ref{p2l3}). 

\medskip

This paper also proved that the nonuniformly complete controllability implies the feedback stabilizability, where the linear feedback gain leads to a nonuniformly exponentially 
stable linear system (Theorem \ref{T2}). We stressed that this class of stability is a limit case of a dichotomy property, namely, the nonuniform exponential dichotomy (NUED), and allows an interpretation of the feedback stabilization from a spectral theory arising from the NUED (Corollary \ref{coro-spectrum}).

\medskip
The current results can certainly be improved and also raise new questions:

\medskip

\noindent a) Our result of feedback stabilization (Theorem \ref{T2}) cannot be achieved for exponential
decays $\mathcal{L}-\theta_{1}\in (0,3\theta_{1}+2\theta_{1}]$. This restriction is due to the specific approach 
carried out in our proof (Proposition \ref{T3}) and could be improved in future research.

\medskip 

\noindent b) Another open problem is to obtain a feedback stabilization for a general
NUCC linear system. Note that an assumption of Theorem \ref{T2} is that the plant has the property of nonuniform bounded growth. It will be extremely interesting to generalized our result by considering the nonuniform Kalman's condition instead of the nonuniform bounded growth.

\medskip 

\noindent c) Our intepretation of Theorem \ref{T2} from a perspective of the nonuniform
exponential dichotomy spectrum  $\Sigma_{NU}(A)$ suggests to enquire about the links between feedback stabilizability and feedback assignability of the spectrum, that is, given a set $\mathcal{S}\subset \mathbb{R}$, 
find a feedback gain $F$ such that $\Sigma_{NU}(A-BF)=\mathcal{S}$. This problem has been previously studied
for UCC nonautonomous systems and we refer to \cite{Anh,Babiarz} for details.

\appendix

\section{Energy of the controlling  input}\label{app:energy}

The energy of an input $u(\cdot)$ of the control system \eqref{control1} is the nonne\-gative number:
$$
E(u)=\left(\int_{t_0}^{+\infty}u^T(t)u(t)dt\right)^{\frac12}.
$$

The following is an analysis from the point of view of the energy required for the input $u^\star(\cdot)$ described by 
 \eqref{input} to take an initial state $x_0$ and bring it to the origin. We have that
$$
\begin{array}{rcl}
E^2(u^\star)&=&\displaystyle\int_{t_0}^{t_f}x_0^TW^{-T}(t_0,t_f)\Phi(t_0,t)B(t)B^T(t)\Phi^T(t_0,t)W^{-1}(t_0,t_f)x_0dt,\\\\
&=&\displaystyle x_0^TW^{-T}(t_0,t_f)\left[\int_{t_0}^{t_f}\Phi(t_0,t)B(t)B^T(t)\Phi^T(t_0,t)dt\right]W^{-1}(t_0,t_f)x_0,\\\\
&=&\displaystyle x_0^TW^{-T}(t_0,t_f)\left[W(t_0,t_f)\right]W^{-1}(t_0,t_f)x_0= x_0^TW^{-T}(t_0,t_f)x_0,
\end{array}
$$
where $W^{-T}$ stands for the transpose of the inverse of $W$. Since $W(t_0,t_f)$ is positive definite, then $W^{-T}(t_0,t_f)$ is also positive definite. In addition,  $W^{-T}(t_0,t_f)$ and $W^{-1}(t_0,t_f)$ have the same eigenvalues and they are the inverse of the eigenvalues of $W(t_0,t_f)$. Based on the above, we have the following bounds:
$$
\lambda_{\min}(W^{-1}(t_0,t_f))\|x_0\|^2
\leq E^2(u^\star) \leq \lambda_{\max}(W^{-1}(t_0,t_f))\|x_0\|^2
$$
or, equivalently:
\begin{equation}
\label{CotasEnergia}
\lambda_{\max}(W(t_0,t_f))\|x_0\|^2
\leq E^2(u^\star) \leq \lambda_{\min}(W(t_0,t_f))\|x_0\|^2.
\end{equation}

In the context of the classical uniform case, from \eqref{UC1} and \eqref{UC2}, with $\sigma=t_f-t_0$, it follows that: 
$$
\alpha_0(\sigma)\leq \lambda_{\min}(W(t_0,t_f))\leq  \lambda_{\max}(W(t_0,t_f))\leq \alpha_1(\sigma),
$$
therefore, by considering this estimation for inequality (\ref{CotasEnergia}) we have that:
$$
\alpha_0(\sigma)\|x_0\|^2\leq E^2(u^\star) \leq \alpha_1(\sigma)\|x_0\|^2
$$
and the required energy, in order to control the state $x_0$, can not be made arbitrarily small. Moreover, in the \emph{worst case scenario}, the required energy is $\alpha_1(\sigma)\|x_0\|^2
$, which is independent of the initial time $t_0$.

Now, when considering the nonuniform complete controllability framework, there exists nonnegative numbers $\{\mu_0, \mu_1, \tilde{\mu}_0, \tilde{\mu}_1\}$ and functions $\{\alpha_0(\cdot), \beta_0(\cdot), \alpha_1(\cdot),\beta_1(\cdot)\}$ such that for any $t_0\geq0$, there exists $\sigma_0(t_0)>0$ where the estimates \eqref{alpha0alpha1} and \eqref{K} are verified for all $\sigma\geq\sigma_0(t_0)$. In particular, from \eqref{alpha0alpha1}, the size of eigenvalues of $W(t_0,t_0+\sigma)$ depends on the initial time $t_0$ and the energy bounds are the following:
$$
e^{-2\mu_0t_0}\alpha_0(\sigma)\|x_0\|^2\leq E^2(u^\star) \leq e^{2\mu_1t_0}\alpha_1(\sigma)\|x_0\|^2.
$$

Let us observe that the required controlling energy belongs to an interval that exponentially grows towards $(0,+\infty)$ as $t_0\to +\infty$. In other words, and considering the same approach as in item b) of Remark \ref{Rmu}, the last inequality indicates that the lower and upper bounds for the energy depend on the initial time $t_0$, which corresponds to a generalization of the uniform case.


\begin{thebibliography}{99}

\bibitem{Anderson67} B.D.O. Anderson, L.M. Silverman.
Uniform complete controllability for time--varying systems.
IEEE Transactions on Automatic Control AC-12:790--791, 1967.

\bibitem{Anderson-1990} B.D.O. Anderson, J.B. Moore.
Optimal Control, Linear Quadratic Methods.
Dover Publications, Minneola NY, 1990.

\bibitem{Anderson} B.D.O. Anderson, A. Ilchmann, F. Wirth.
Stabilizability of linear time--varying systems.
System $\&$ Control Letters 62:747--755, 2013.

\bibitem{ZA} Z. Artstein. Uniform Controllability via the limiting systems.
Appl. Math. Optim. 9:111--131, 1982.

\bibitem{Anh}
P. Anh, A. Czornik, T.S. Doan, S. Siegmund.
Proportional local assignability of dichotomy spectrum of one-sided continuous time--varying linear systems.
J. Differential Equations 309: 176–195, 2022.

\bibitem{Barreira} L. Barreira, C. Valls,
Stability of Nonautonomous Differential Equations.
Springer--Verlag, Berlin--Heidelberg, 2008. 

\bibitem{Babiarz} A. Babiarz, L.V. Cuong, A. Czornik, T.S. Doan.
Necessary and sufficient conditions for assignability of dichotomy spectra of continuous time--varying linear systems. 
Automatica J. 125, Paper No. 109466, 2021.


\bibitem{Coppel} W. Coppel,
Dichotomies in stability theory.
Springer--Verlag,  Berlin--Heidelberg, 1978.

\bibitem{Chen} C.T. Chen,
Introduction to Linear Systems Theory,
Holt, Rinehart and Winston Inc., New York, 1970.

\bibitem{Chu} J. Chu, F.F. Liao, S. Siegmund, Y. Xia, W. Zhang,
Nouniform dichotomy spectrum and reducibility for nonautomomous equations,
Bull. Sci. Math. 139:538--557, 2015.


\bibitem{CLSD} B. D'Andrea--Novel, M. Cohen de Lara.
Commande Lin\'eaire des Syst\'emes Dynamiques.
Masson, Paris, 1993.


\bibitem{Hahn} W. Hahn.
Stability of Motion, springer--Verlag, Berlin, 1967.

\bibitem{Ikeda} M. Ikeda,  H. Maeda, S. Kodama.
Stabilization of linear systems.
SIAM J. Control 10:716--729, 1972.

\bibitem{Ilchmann} A. Ilchmann, G. Kern.
Stabilizability of systems with exponential dichotomy.
System $\&$ Control Letters 8:211--220, 1987.


\bibitem{Kalman-60} R. Kalman.
On the general theory of control systems. In
Proceedings of the first IFAC congress, Vol.1,
Butterworth's, London, pp. 481--491, 1960.

\bibitem{Kalman} R. Kalman. 
Contributions to the theory of optimal control. 
Bol. Soc. Mat. Mexicana 5:102--119, 1960.

\bibitem{Kalman69} R. Kalman.
Lectures on Controllability and Observability.
Centro Internazionale Matematico Estivo (CIME),
Edizioni Cremonese, Roma, 1969.


\bibitem{Klo} P.E. Kloeden and M. Rasmussen: 
Nonautonomous Dynamical Systems, Mathematical Surveys and Monographs, 
Volume 176, AMS, Providence RI (2011).

\bibitem{Kreindler} E. Kreindler, P.E. Sarachik.
On the concepts of controllaility and observability
of linear systems.
IEEE Transactions on Automatic Control, AC-10:129--136, 1964.

\bibitem{Liao} F. Liao, Y. Jiang, Z. Xie. A Generalized Nonuniform Contraction and Lyapunov Function. Abstr. Appl. Anal. 2012: 1-14, 2012.


\bibitem{Lin2} Z. Lin and Y-X. Lin.
Linear Systems Exponential Dichotomy and Structure of Sets of 
Hyperbolic Points, World Scientific, Singapore (2000).


\bibitem{Mitropolski} Y.A. Mitropolsky, A.M. Samoilenko and
V.L. Kulik.
Dichotomies and stability in nonautonomous linear systems,
Taylor and Francis, London -- New York (2003).


\bibitem{Palmer} K.J. Palmer.
Exponential dichotomy and expansivity.
Annali di Matematica 185:S171--S185, 2006.

\bibitem{Phat} V.N. Phat, Q.P. Ha.
New characterization of controllability via stabilizability and Riccati equation for LTV systems.
IMA J. Math. Control Inf. 11:419--429, 2008.

\bibitem{Rotea} M.A. Rotea, P. Khargonekar.
Stabilizability of linear time--varying and uncertain linear systems.
IEEE Transactions on Automatic Control 33:884--887, 1988.

\bibitem{Rouche} N. Rouche, P. Habets, M. Laloy.
Stability Theory by Lyapunov's Direct Methods. 
Springer--Verlag, New York, 1977.

\bibitem{Rugh} W.J. Rugh.
Linear System Theory.
Prentice Hall, Upper Saddle River NJ, 1996.

\bibitem{Siegmund-2002} S. Siegmund:
Dichotomy spectrum for nonautonomous differential equations,
J. Dynam. Differential Equations 14:243--258, 2002. 

\bibitem{Siegmund2}
S. Siegmund: Reducibility of nonautonomous linear differential equations. J. London Math. Soc. (2) 65 (2002), 397–410.



\bibitem{Silva2023} C. M. Silva. Nonuniform $\mu-$dichotomy spectrum and kinematic similarity. 
Journal of Differential Equations 375:618--652, 2023. 


\bibitem{Silverman} L.M. Silverman, B.D.O. Anderson.
Controllability, Observability and stability of linear systems.
SIAM J. Control 6: 121--130, 1968.

\bibitem{Sontag} E. Sontag.
Mathematical Control Theory,
Deterministic Finite Dimensional Systems.
Springer, New York, 1998.



\bibitem{Zhang} X. Zhang,
Nonuniform dichotomy spectrum and normal forms for nonautonomous differential systems, J. Funct. Anal. 267:1889--1916, 2014.

\bibitem{ZDG} K. Zhou, J.C. Doyle, K. Glover.
Robust and Optimal Control,
Prentice Hall, Upper Saddle River NJ, 1996.

\bibitem{Zhou-16} B. Zhou.
On asymptotic stability of linear time--varying systems,
Automatica 68:266--276, 2016.

\bibitem{Zhou-17} B. Zhou.
Stability analysis of non--linear time varying systems
by Lyapunov functions with indefinite derivatives,
IET Control Theory $\&$ Applications 11:1434--1442, 2017.

\bibitem{Zhou-21} B. Zhou.
Lyapunov differential equations and inequalities for stability and stabilization of linear 
time--varying systems, 
Automatica 131, Paper No. 109785, 2021.

\bibitem{Zhou-JDE} L. Zhou, K. Lu, W. Zhang.
Equivalences between nonuniform exponential dichotomy and admissibility.
J. Differential Equations 262:682--747, 2017.


\bibitem{Zhu} H. Zhu, Z. Li.
Nonuniform dichotomy spectrum intervals: theorem and computation,
J. Appl. Anal. Comput. 9:1102--1119, 2019.

\end{thebibliography}
\end{document}